\documentclass[a4paper,twoside,11pt]{article}

\usepackage{BeaCed_article}

\begin{document}

\title{The critical $Z$-invariant Ising model via dimers:\\ the periodic case}
\author{C\'edric Boutillier
\thanks{
{\small Laboratoire de Probabilit\'es et Mod\`eles Al\'eatoires, Universit\'e Paris VI Pierre et Marie Curie, Case courrier 188, 4 place Jussieu, F-75252 Paris CEDEX 05.} \small\texttt{cedric.boutillier@upmc.fr}. Supported in part by the Swiss National Foundation Grant 200020-120218/1.
}
, B\'eatrice de Tili\`ere
\thanks{
{\small Institut de Math\'ematiques, Universit\'e de Neuch\^atel, Rue Emile-Argand 11, CH-2007 Neuch\^atel.}
{\small\texttt{beatrice.detiliere@math.unizh.ch}}.
{\small Supported in part by the Swiss National Foundations grants 47102009 and 200020-120218/1.}
} }
\date{}
\maketitle
\vspace{-1cm}
\begin{abstract}

  We study a large class of critical two-dimensional Ising models namely {\em
  critical $Z$-invariant Ising models on periodic graphs}, example
  of which are the classical $\ZZ^2$, triangular and honeycomb
  lattice at the critical temperature. Fisher \cite{Fisher}
  introduced a correspondence between the Ising model and the dimer
  model on a decorated graph, thus setting dimer techniques as a
  powerful tool for understanding the Ising model. In this paper, we
  give a full description of the dimer model corresponding to the
  critical $Z$-invariant Ising model. We prove that the dimer
  characteristic polynomial is equal (up to a constant) to the
  critical Laplacian characteristic polynomial, and defines a
  Harnack curve of genus $0$. We prove an explicit expression for
  the free energy, and for the Gibbs measure obtained as weak limit
  of Boltzmann measures.
\end{abstract}

\section{Introduction}

In \cite{Fisher}, Fisher introduced a correspondence between the two-dimensional Ising model defined on a graph $G$, and the dimer model defined on a decorated version of this graph. Since then, dimer techniques have been a powerful tool for solving pertinent questions about the Ising model, see for example the paper of Kasteleyn \cite{Kasteleyn}, and the book of Mc Coy and Wu \cite{McCoyWu}. In this paper, we follow this approach to the Ising model.

We consider a large class of critical Ising models, known as {\em critical $Z$-invariant Ising models}, introduced in \cite{Baxter}, studied in \cite{Mercat,CostaSantos}. More precisely, we consider Ising models defined on graphs which have the property of having an {\em isoradial embedding}. We suppose that the Ising coupling constants naturally depend on the geometry of the embedded graph, and are such that the model is invariant under {\em star-triangle} transformations of the underlying graph, {\em i.e.} such that the Ising model is {\em $Z$-invariant}. We suppose moreover that the coupling constants are {\em critical} by imposing a {\em generalized self-duality property}. The standard Ising model on the square, triangular and honeycomb lattice at the critical temperature are examples of critical $Z$-invariant Ising models.

Let $G$ be an infinite $\ZZ^2$-periodic isoradial graph with critical coupling constants on the edges. Then, by Fisher the Ising model on $G$ is in correspondence with the dimer model on a decorated graph $\GD$, with a well chosen positive weight function $\nu$ on the edges. The decorated graph $\GD$ is also infinite and $\ZZ^2$-periodic. A natural exhaustion of the infinite graph $\GD$ by toroidal graphs is given by $\{\GD_n\}_{n\geq 1}$, where $\GD_n=\bigl(V(\GD_n),E(\GD_n)\bigr)= \GD / n\ZZ^2$. The dimer model on $\GD_n$ is defined as follows. A {\em dimer configuration} $M$ of $\GD_n$ is a subset of edges of $\GD_n$ such that every vertex is incident to exactly one edge of $M$. Let $\M(\GD_n)$ be the set of dimer configurations of $\GD_n$. The probability of occurrence of a dimer configuration $M$ is given by the {\em dimer Boltzmann measure}, denoted $\P_n$:
\begin{equation*}
\P_n(M)=\frac{ \prod_{e\in M}\nu_e}{\Z_n},
\end{equation*}
where $\Z_n=\sum_{M\in\M(\GD_n)}\prod_{e\in M}\nu_e$ is the {\em dimer partition function}. 

The first theorem of this paper proves an explicit expression for the free energy, and for the Gibbs measure obtained as weak limit of Boltzmann measures, of the dimer model on the infinite graph $\GD$. It can loosely be stated as follows, a precise statement is given in Theorem \ref{prop:periodic} of Section \ref{sec:periodic}. Let $K$ be a {\em Kasteleyn matrix} of the graph $\GD$. It is a weighted, oriented adjacency matrix of the graph $\GD$, defined in Section \ref{sec:Kmatrix}.
\begin{thm}$\,$\label{thm:thm1}
\begin{enumerate}
 \item The free energy of the dimer model on $\GD$ is:
\begin{equation*}
-\lim_{n\rightarrow\infty}\frac{1}{n^2}\log(\Z_n)=-\frac{1}{2(2\pi i)^2}\iint_{\TT^2}\log \det(\widehat{K}(z,w))\frac{\ud z}{z}\frac{\ud w}{w},
\end{equation*}
where $\widehat{K}(z,w)$ is the Fourier transform of the infinite Kasteleyn matrix $K$ (it is a matrix of size $V(\GD_1)\times V(\GD_1)$).
\item The weak limit of the Boltzmann measures $\P_n$ defines a translation invariant ergodic Gibbs measure $\P$. The probability of occurrence of a subset of edges $e_1=u_1 v_1,\cdots,e_m=u_m v_m$ of $\GD$, in a dimer configuration of $\GD$ chosen with respect to the Gibbs measure $\P$, is:
\begin{equation*}
\P(e_1,\cdots,e_m)=\left(\prod_{i=1}^m K_{u_i,v_i}\right)\Pf((K^{-1})_{e_1,\cdots,e_m}),
\end{equation*}
where $K^{-1}$ is the inverse of the infinite Kasteleyn matrix defined in \eqref{eq:Kinverse}, and
$(K^{-1})_{e_1,\cdots,e_m}$ is the submatrix of $K^{-1}$, whose lines and columns are indexed by the vertices defining the edges $e_1,\cdots,e_m$. 
\end{enumerate}
\end{thm}

The statement of Theorem \ref{thm:thm1} is similar to Theorems 3.5 and 4.3 of \cite{KOS}, see also \cite{CKP}, obtained for dimer models defined on infinite $\ZZ^2$-periodic {\em bipartite} graphs. These results do not apply as such to our setting, where the graph $\GD$ on which our dimer model is defined, is {\em not} bipartite. Here is a description of the additional problems we had to solve when dealing with a non-bipartite graph.
\begin{enumerate}
 \item A key requirement in proving such a theorem is to have a precise description of the zeros of $\det(\widehat{K}(z,w))$. In the bipartite case, Kenyon, Okounkov and Sheffield, prove that this zero set is a Harnack curve, which implies that there are at most two zeros on $\TT^2$. Quite surprisingly, we also prove that in our case, the zero set is a Harnack curve, by relating it to the zero set of the Laplacian. This is the content of Theorem \ref{thm:thm2} below.
\item In order to prove weak convergence of the Boltzmann measures, Kenyon, Okounkov and Sheffield prove weak convergence of the measures on a sub-sequence of $n$'s, which is quite easy once a description of the zero set of $\det(\widehat{K}(z,w)))$ is available, and use the convergence theorem for Gibbs measures of Sheffield \cite{Sheffield} to deduce convergence for every $n$. When the graph is not bipartite, Sheffield's theorem does not hold. We give a direct argument proving weak convergence of the Boltzmann measures for every $n$. Our techniques are general and robust, so that they could be used for dealing with dimer models on other non bipartite graphs.
\end{enumerate}

The description of the zero set of $\det(\widehat{K}(z,w))$ is given by Theorem \ref{thm:thm2} below. An equivalent statement is given in Theorem \ref{thm:P_Harnack} of Section~\ref{sec:P_Harnack}, where it is also proved. Let $\Delta$ be the Laplacian matrix of the infinite $\ZZ^2$-periodic isoradial graph $G$, and let $\widehat{\Delta}(z,w)$ be its Fourier transform. 
\begin{thm}
  \label{thm:thm2}
There exists a constant $c\neq 0$, depending on the isoradial embedding of the graph $G$, such that: 
\begin{equation*}
\det(\widehat{K}(z,w))=c \det(\widehat{\Delta}(z,w)),
\end{equation*}
and $\{(z,w)\in \CC^2\ :\ \det(\widehat{K}(z,w))=0\}$ is a Harnack curve of genus $0$. Its unique  point on $\TT^2$ is $(1,1)$, and it has multiplicity $2$.
\end{thm}
The main steps of the proof are:
\begin{enumerate}
\item Show that the zero set of $\det(\widehat{\Delta}(z,w))$ is the spectral curve of the dimer model on an isoradial bipartite graph. The main tool to do this is an interpretation of $\det \widehat{\Delta}(z,w)$ in terms of \emph{cycle-rooted spanning forests}, which is an analog of Kirchhoff's matrix-tree theorem~\cite{Kirchhoff:1847}. Then use the result of Kenyon and Okounkov \cite{KO} by which spectral curves of dimer models on bipartite isoradial graphs are Harnack curves of genus $0$.
\item Show that $\det(\widehat{K}(z,w))=c \det(\widehat{\Delta}(z,w))$. First prove that $\det(\widehat{K}(z,w))$ divides $\det(\widehat{\Delta}(z,w))$ by using an explicit parametrization of the zero set of $\det(\widehat{\Delta}(z,w))$, then use the interpretation of $\det(\widehat{\Delta}(z,w))$ as cycle-rooted spanning forests to show inclusion of the Newton polygons.
\end{enumerate}

{\bf Outline of the paper}
\begin{description}
\item[Section $2$:] Definition of the critical $Z$-invariant Ising model.
\item[Section $3$:] Fisher's correspondence between the Ising and dimer models.
\item[Section $4$:] Study of the critical dimer model corresponding to a critical $Z$-invariant Ising model on a $\ZZ^2$-periodic graph. Definition of the infinite Kasteleyn matrix~$K$. Proposition \ref{prop:uniqueness} proves the existence and uniqueness of the inverse of the matrix~$K$. Statement and proof of Theorem \ref{prop:periodic}.
\item[Section $5$:] Statement and proof of Theorem \ref{thm:P_Harnack}.
\end{description}

\section{The critical $Z$-invariant Ising model}

Consider an unoriented finite graph $G=(V(G),E(G))$, together with a collection of positive real numbers $J=(J_e)_{e\in E(G)}$ indexed by the edges of $G$. The {\em Ising model on $G$ with coupling constants $J$}, is defined as follows.

A \emph{spin configuration} $\sigma$ of $G$ is a function of the vertices of $G$ with values in $\{-1,+1\}$. The probability of occurrence of a spin configuration $\sigma$ is given by the {\em Ising Boltzmann measure}, denoted $P^J$:
\begin{equation*}
P^J(\sigma)=\frac{1}{Z^J}\exp\biggl(\sum_{e=uv\in E(G)}J_e\sigma_u\sigma_v\biggr),
\end{equation*}
where 
\begin{equation*}
 Z^J=\sum_{\sigma\in\{-1,1\}^{V(G)}}\exp\biggl(\sum_{e=uv\in E(G)}J_e\sigma_u\sigma_v\biggr),
\end{equation*}
is the {\em Ising partition function}. 

\subsection{The $Z$-invariant Ising Model}

$Z$-invariant Ising models are Ising models defined on a class of embedded graphs which have the property of being {\em isoradial}. The corresponding coupling constants are naturally related to geometric features of these graphs, and satisfy the {\em star-triangle relation} defined below.

\subsubsection{Isoradial graphs}

A graph $G$ is said to be {\em isoradial}~\cite{Kenyon3,KeSchlenk}, if it has an embedding 
in the plane such that every face is inscribed in a circle of radius~1, and
all circumcenters of the faces are in the closure of the faces. From now on, when we speak 
of the graph
$G$, we mean the graph together with a particular isoradial embedding in the
plane. Examples of isoradial graphs are given in Figure \ref{fig:isingcrit} below.

\begin{figure}[ht]
  \begin{center}
    \includegraphics[width=\linewidth]{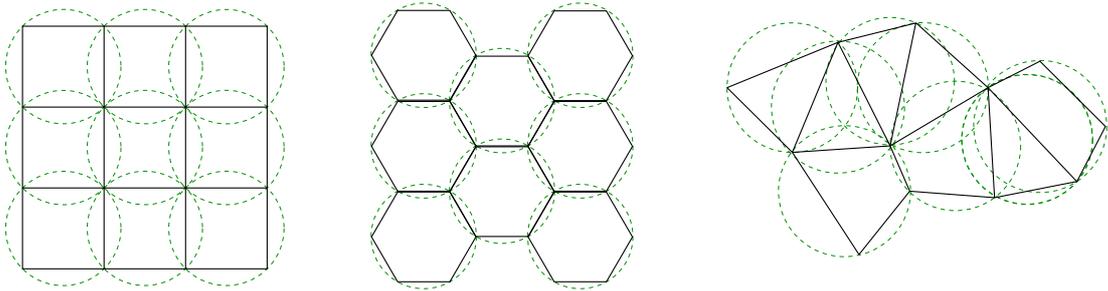}
    \caption{Examples of isoradial graphs: the square lattice (left), the honeycomb lattice (center), and a more generic one (right). Every face is inscribed in a circle of radius 1, represented in dashed lines.}
    \label{fig:isingcrit}
  \end{center}
\end{figure}

To such a graph is naturally associated the {\em diamond graph}, denoted by $\GR$, defined as follows. Vertices of $\GR$ consist in the vertices
of $G$, and the circumcenters of the faces of $G$. The circumcenter of each
face is then joined to all vertices which are on the boundary of this face, see
Figure \ref{fig:iso-rhombi}. Since $G$ is isoradial,
all faces of $\GR$ are side-length-$1$ rhombi. Moreover, each edge $e$ of
$G$ is the diagonal of exactly one rhombus of $\GR$; we let $\theta_e$ be the
half angle of the rhombus at the vertex it has in common with $e$, see Figure \ref{fig:iso-rhombi} (right).

\begin{figure}[h]
  \begin{center}
    \includegraphics[width=\linewidth]{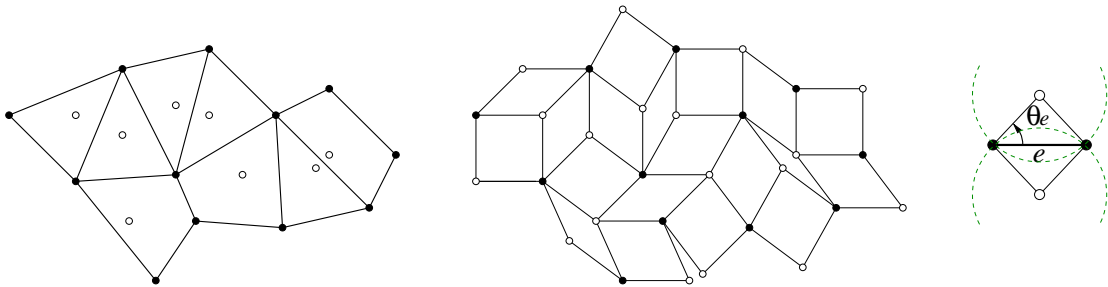}
    \caption{An isoradial graph (left). The white dots are the circumcenters of the faces, that are also the vertices of the dual $G^*$. Its diamond graph is represented in the center. On the right is the half-rhombus angle $\theta_e$ assigned to an edge $e$.}
    \label{fig:iso-rhombi}
  \end{center}
\end{figure}

The same construction can be done for \emph{toroidal isoradial graphs} in which case the embedding is on a torus.

\subsubsection{Star-triangle transformation}

In this section, we suppose that $G$ is a toroidal isoradial graph. It is then natural to choose the coupling constants $J$ of the Ising model on $G$, to depend on the geometry of the embedded graph: let us assume that $J$ is a function of $\theta_e$, the rhombus half-angle assigned to the edge $e$. 

The {\em star-triangle transformation} is a local operation on the graph $G$ which preserves isoradiality.
If $G$ has a vertex of degree $3$ (star), it can be removed and replaced by a triangle reconnecting its neighbors, see Figure \ref{fig:star-triangle}. The Ising model is said to be \emph{$Z$-invariant} if the Boltzmann probability of any event not involving the spin of the degree-3 vertex is preserved under this operation.

\begin{figure}[ht]
  \begin{center}
    \includegraphics[height=2.8cm]{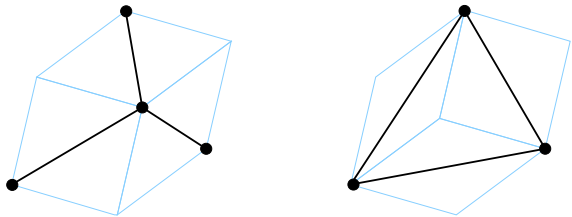}
    \caption{The star-triangle transformation corresponds to a rearrangement of three rhombi of the diamond graph. Isoradiality is preserved by this transformation.}
    \label{fig:star-triangle}
  \end{center}
\end{figure}

Imposing this condition is a strong constraint. Indeed, there is only a one parameter family of coupling constants which defines a $Z$-invariant Ising model: $J(\theta_e)$ must satisfy the equation:
\begin{equation}\label{eq:iso_weight}
  \sinh(2 J(\theta_e))= \frac{\mathrm{sn}\bigl(\frac{2K(k)}{\pi}\theta_e\big|k\bigr)}{\mathrm{cn}\bigl(\frac{2K(k)}{\pi}\theta_e\big|k\bigr)},\quad k^2\in\RR.
\end{equation}
The functions $\mathrm{sn}$, $\mathrm{cn}$ are Jacobi elliptic sine and cosine\footnote{ See \cite{GradsteinRyshik}, Section 6.1 for definitions and properties of these functions.}, and $K$ is the complete integral of the first kind:
\begin{equation*}
  K(k)=\int_{0}^{\pi/2} \frac{\mathrm{d}\varphi}{\sqrt{1-k^2 \sin^2\varphi}}.
\end{equation*}

A detailed derivation of this parametrization is given in \cite{Baxter}, Chap. 6.

\subsection{Polygonal contour expansions of the Ising partition function}

Let $G$ be a toroidal graph, with coupling constants $J$, not necessarily isoradial. There are two ways to define a probability measure on polygonal contours from the Ising model on $G$: 
the so-called \emph{low-temperature} and \emph{high-temperature}
expansions of the partition function. 

The \emph{high temperature expansion} is based on the following identity:
\begin{equation*}
 e^{J_e\sigma_u\sigma_v}=\cosh(J_e)\left(1+\tanh(J_e)\sigma_u\sigma_v\right).
\end{equation*}

The partition function can be rewritten as
\begin{align}\label{eq:Zhightemp}
 Z^J&=\biggl(\prod_{e\in E(G)}\cosh(J_e)\biggr)
\sum_{\sigma\in\{-1,1\}^{V(G)}} \prod_{e=uv\in E(G)}\left(1+\sigma_u\sigma_v \tanh(J_e)\right)\nonumber\\
&=\biggl(\prod_{e\in E(G)}\cosh(J_e)\biggr)2^{|V(G)|}\sum_{\C\in\mathcal{P}} \prod_{e\in\C} \tanh(J_e),
\end{align}
where $\mathcal{P}$ is the family of all polygonal contours drawn on $G$, for
which every edge of $G$ is used at most once. 

The {\em low-temperature expansion} gives a probability measure on polygonal
contours on the dual graph $G^*$ of $G$, which represent the separations between the clusters of $+$ spins and those of $-$ spins. This expansion is based on the following identity:
\begin{equation*}
 e^{ J_e\sigma_u\sigma_v}=e^{J_e}\left(1+\frac{1-\sigma_u\sigma_v}{2}e^{-2J_e}\right).
\end{equation*}
The partition function is thus equal to
\begin{align}\label{eq:Zlowtemp}
 Z^J&=\biggl(\prod_{e\in E(G)}
   \exp(J_e)\biggr)\sum_{\sigma\in\{-1,1\}^{V(G)}} \prod_{e=uv\in E(G)}\biggl(1+\frac{1-\sigma_u\sigma_v}{2}e^{-2J_e}\biggr)\nonumber\\
&=\biggl(\prod_{e\in E(G)} \exp(J_e)\biggr)\sum_{\C^*\in\mathcal{P^*}} \prod_{e^*\in\C^*}\exp(-2J_e),
\end{align}
where $\mathcal{P}^*$ is the family of polygonal contours drawn on $G^*$, 
for which every edge of $G^*$ is used at most once.

In both cases, the probability of a polygonal contour configuration is defined to be proportional to the product of the weights of the edges it contains. The weight of an edge $e$ is $\tanh(J_e)$ in the high temperature expansion, and $\exp(-2J_{e^*})$ in the low temperature expansion.

\subsection{The critical $Z$-invariant Ising model\label{sec:critZinv}}

Consider a $Z$-invariant Ising model on a toroidal isoradial graph $G$, and a $Z$-invariant Ising model on its dual graph $G^*$. Let $J$ and $J^*$ be the associated coupling constants satisfying \eqref{eq:iso_weight} for some parameters $k$ and $k^*$ respectively. Then, if $k^*=\frac{i k}{\sqrt{1-k^2}}$, the measure on polygonal contours of $G$ obtained from the high temperature expansion on $G$, is the same as the measure on polygonal contours of $G$ obtained from the low temperature expansion on $G^*$: there is a duality relation between $Z$-invariant Ising models with parameters $k$ and $k^*$. It differs from the classical one on the square lattice~\cite{KramersWannier} in the sense that the graph $G$ and its dual are \emph{a priori} not related: this duality relation must be viewed on the class of all $Z$-invariant Ising models.

The only value of $k$ for which the set of $Z$-invariant Ising models with parameter $k$ maps onto itself (that is $k^*=k$), is $k=0$. At this special value, we say that the model satisfies a {\em generalized self-duality property}. The expression for the weights simplify drastically, indeed:

\begin{equation*}
  K(0)=\frac{\pi}{2}, \quad \mathrm{sn}\biggl(\frac{2 K(0)}{\pi}\theta\bigg\vert 0\biggr)= \sin \theta,\quad  \mathrm{cn}\biggl(\frac{2 K(0)}{\pi}\theta\bigg\vert 0\biggr)= \cos \theta.
\end{equation*}

The expression of the coupling constants reduces to:
\begin{equation}\label{eq:critical_coupling}
  J(\theta_e) = \frac{1}{2} \log \left(\frac{1+\sin \theta_e}{\cos \theta_e}\right).
\end{equation}
In the case of the square lattice, the rhombus half-angle assigned to each edge is $\pi/4$, so that the coupling constants are all equal to $\log\sqrt{1+\sqrt{2}}$, which is the value of the critical temperature of the Ising model on the square lattice \cite{KramersWannier}. Similarly, on the honeycomb (resp. triangular) lattice, the rhombus half-angle assigned to each edge is $\pi/3$ (resp. $\pi/6$), in this case as well the coupling constants coincide with the critical temperature $\log\sqrt{2+\sqrt{3}}$ (resp. $\log \sqrt{\sqrt{3}}$) \cite{Wannier}. For these reasons, we refer to $k=0$ as the \emph{critical value} of the parameter, and the choice of weights \eqref{eq:critical_coupling} as the \emph{critical coupling constants}.

The $Z$-invariant Ising model on an isoradial graph with this particular choice of coupling constants is referred to as the \emph{critical $Z$-invariant Ising model}.

\section{Fisher's correspondence between the Ising and Dimer models\label{sec:isingdimers}}

Let $G$ be any graph drawn on a torus, with coupling constants $J$. Then, Fisher \cite{Fisher} exhibits a correspondence between the Ising model on the
graph $G$, and the dimer model on a ``decorated'' version of $G$. This correspondence uses the high temperature expansion of the Ising partition function given in \eqref{eq:Zhightemp}. Let us first recall the definition of the dimer model.

\subsection{Dimer model}

Consider a finite graph $\GD=(V(\GD),E(\GD))$, and suppose that edges of $\GD$ are assigned a positive weight function $\nu=(\nu_e)_{e\in E(\GD)}$. The {\em dimer model on $\GD$ with weight function $\nu$} is defined as follows.

A {\em dimer configuration} $M$ of $\GD$, also called {\em perfect matching}, is a subset of edges of $\GD$ such that every vertex is incident to exactly one edge of $M$. Let $\M(\GD)$ be the set of dimer configurations of the graph $\GD$. The probability of occurrence of a dimer configuration $M$ is given by the {\em dimer Boltzmann measure}, denoted $\P^\nu$:
\begin{equation*}
 \P^\nu(M)=\frac{\prod_{e\in M}\nu_e}{\Z^\nu},
\end{equation*}
where $\Z^\nu=\sum_{M\in \M(\GD)}\prod_{e\in M}\nu_e$ is the {\em dimer partition function}.

\subsection{Fisher's correspondence}

Consider an Ising model on a toroidal graph $G$, with coupling constants $J$. In this paper we use the following slight variation of Fisher's correspondence.

The decorated graph, on which the dimer configurations live, is constructed from $G$ as follows. Every vertex of degree $k$ of $G$ is replaced by a {\em
decoration} consisting of $3k$ vertices: a triangle is attached to
every edge incident to this vertex, and these triangles are linked by edges in
a circular way, see Figure \ref{fig:decorated_graph} below. This new graph, denoted by $\GD$, is also embedded on the torus and has vertices of degree $3$. It is referred to as the {\em Fisher graph} of $G$. 

\begin{figure}[ht]
\begin{center}
\includegraphics[height=33mm]{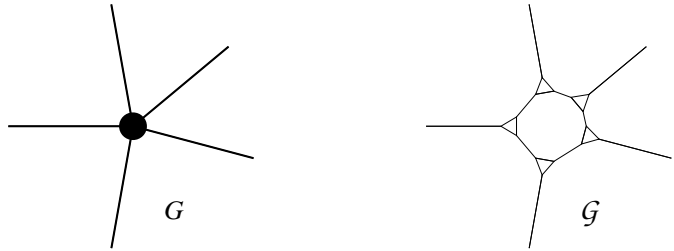}
\caption{A vertex of $G$ (left) with its incoming edges, and the corresponding decoration in $\GD$ (right).}
\label{fig:decorated_graph}
\end{center}
\end{figure}

Here comes the correspondence: to any contour configuration $\C$ coming from the high-temperature expansion of the Ising model on $G$, we associate $2^{|V(G)|}$ dimer configurations on $\GD$: edges present (resp. absent) in $\C$ are
absent (resp. present) in the corresponding dimer configuration of $\GD$. Once
the state of these edges is fixed, there is, for every decorated vertex,
exactly two ways to complete the configuration into a dimer configuration.
Figure \ref{fig:Fisher_correspondence} below gives an example in the case where $G$ is the square lattice $\ZZ^2$.

\begin{figure}[ht]
\begin{center}
\includegraphics[width=\linewidth]{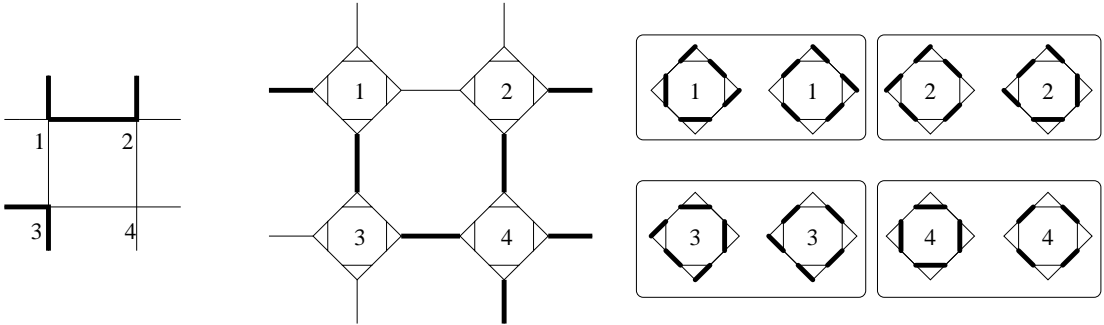}
\caption{Polygonal contour of $\ZZ^2$, and corresponding dimer configurations
  of the associated Fisher graph.}
\label{fig:Fisher_correspondence}
\end{center}
\end{figure}

Let us assign, to an edge $e$ of $\GD$, weight $\nu_e=1$, if it belongs to a
decoration; and weight $\nu_e=\coth(J_e)$, if it corresponds to an edge of $G$.
Then the correspondence is measure-preserving: every contour configuration $\C$
has the same number ($2^{|V(G)|}$) of images by this correspondence, and the product of the weights of the edges in $\C$, $\prod_{e\in \C} \tanh(J_e)$
is proportional to the weight $\prod_{e\not\in\C} \coth(J_e)$  of any of its corresponding dimer configurations with a proportionality factor, $\prod_{e\in E(G)} \tanh(J_e)$, which is independent of $\C$.

Note that Fisher's correspondence is in fact valid for any graph $G$ drawn on a surface without boundary, in particular, for an infinite planar graph.

\subsection{Critical dimer model on Fisher graphs}

Consider a critical $Z$-invariant Ising model on a toroidal isoradial graph $G$. 
Then, the dimer weights of the corresponding dimer model on the Fisher graph $\GD$ are:
\begin{equation*}
  \nu_e=\begin{cases}
    1 & \text{if $e$ belongs to a decoration,}\\
    \coth(J(\theta_e))=\coth\left(\log\sqrt{\frac{1+\sin\theta_e}{\cos\theta_e}}\,\right)=\cot\left(\frac{\theta_e}{2}\right) & \text{if $e$ comes from an edge of $G$.}
  \end{cases}
\end{equation*}
We refer to these weights as {\em critical dimer weights}, and to the corresponding dimer model as {\em critical dimer model on the Fisher graph $\GD$}.

Because of the self-duality property of the critical $Z$-invariant Ising model, the same dimer model is obtained when considering Fisher's correspondence for the low temperature polygonal contour expansion of the Ising model on the dual graph $G^*$.

\begin{rem}{\em
{\em Critical} dimer models on Fisher graphs do not enter in the framework of {\em critical} dimer models introduced in \cite{Kenyon3}. Indeed, the latter are restricted to live on bipartite isoradial graphs, whereas Fisher graphs are neither bipartite nor isoradial (although they are constructed from isoradial graphs).}
\end{rem}

In Section \ref{sec:periodic} below, we consider critical dimer models on infinite $\ZZ^2$-periodic Fisher graphs. The main results are Proposition \ref{prop:uniqueness} proving existence and uniqueness of the inverse Kasteleyn matrix, and Theorem \ref{prop:periodic} giving an explicit expression for the free energy and for the Gibbs measure obtained as weak limit of Boltzmann measures.

\section{Critical dimer model on infinite periodic Fisher graphs}\label{sec:periodic}

Let $\GD=(V(\GD),E(\GD))$ be an infinite $\ZZ^2$-periodic Fisher graph obtained from an infinite isoradial $\ZZ^2$-periodic graph $G$. Assume that edges of $\GD$ are assigned the dimer critical weight function denoted by $\nu$. 

Consider the exhaustion $\{\GD_n\}_{n\geq 1}$ of $\GD$ by toroidal graphs, where
$\GD_n=\GD/n\ZZ^2$. The graph $\GD_1$ is called the {\em fundamental domain}
of $\GD$. Let us assume that $\GD_1$ has an even number of vertices, (if this is not the case, the graph $\GD_1$ has no dimer configuration).

In Section \ref{sec:Kmatrix}, we define Kasteleyn matrices and state Theorem \ref{thm:torus} which is the yet classical result proving an explicit expression for the partition function of $\GD_n$, denoted by $\Z_n$, and for the the Boltzmann measure of $\GD_n$, denoted by $\P_n$. Theorem \ref{thm:torus} is valid in greater generality, {\em i.e.} for dimer models on $\ZZ^2$-periodic graphs with positive weight functions. We nevertheless choose to state them in the context which is relevant to this paper.

\subsection{Dimer model on toroidal graphs $\GD_n$}
\label{sec:Kmatrix}

The key objects used to obtain explicit expressions for the dimer model are {\em Kasteleyn matrices}. They are weighted oriented adjacency matrices of the graph $\GD_n$ defined as follows.

A {\em Kasteleyn orientation} of the edges of $\GD_n$ is an orientation of the edges such that all 
elementary cycles are {\em clockwise odd}, i.e. when
traveling clockwise around the edges of any elementary cycle of $\GD_n$, the
number of co-oriented edges is odd. If the number of vertices of the graph is even, then such an orientation always exists \cite{Kasteleyn,CimaReshe1}. The {\em Kasteleyn matrix} corresponding to such an orientation is a $|V(\GD_n)|\times |V(\GD_n)|$ skew-symmetric matrix, denoted $K_n$, defined by:
\begin{equation*}
(K_n)_{u,v}=
\begin{cases}
 \nu_{uv}&\text{ if } u\sim v,\text{ and } u\rightarrow v\\
-\nu_{uv}&\text{ if } u\sim v,\text{ and } u\leftarrow v\\
0&\text{ else}.
\end{cases}
\end{equation*}

Let $\gamma_{x,n}$ (resp. $\gamma_{y,n})$ be a path in the
dual graph of $\GD_n$, winding once around the torus horizontally
(resp. vertically). For $\theta,\tau\in\{0,1\}$, let $K_n^{\theta\tau}$ be the
Kasteleyn matrix $K_n$ in which the weights of the edges crossing the horizontal
cycle $\gamma_{x,n}$ are multiplied by $(-1)^{\theta}$, and those crossing the
vertical cycle $\gamma_{y,n}$ are multiplied by $(-1)^{\tau}$. Then, it is always possible to choose the Kasteleyn orientation of the edges of $\GD_n$ to be periodic, and such that Theorem \ref{thm:torus} below holds (changing the Kasteleyn orientation would merely change the order of the linear combinations below).

\begin{thm}{\rm \cite{Kast61,CKP,Russes,Tesler, CimaReshe1},\cite{Ke:LocStat}}\label{thm:torus}
\begin{enumerate}
\item The dimer partition function $\Z_n$ of the graph $\GD_n$ is:
\begin{equation}\label{eq:parttorus}
\Z_n=\frac{1}{2}\left(-\Pf(K_n^{00})+\Pf(K_n^{10})+\Pf(K_n^{01})+\Pf(K_n^{11})\right).
\end{equation}
\item
Let $E=\{e_1=u_1 v_1,\cdots,e_m=u_m v_m\}$, be a subset of edges of $\GD_n$, then the
probability $\P_n(e_1,\cdots,e_m)$ of these edges occurring in a dimer configuration of $\GD_n$ chosen
with respect to the Boltzmann measure $\P_n$, is:
\begin{equation}\label{eq:bolztorus}
\frac{\prod_{i=1}^m K_{u_i,v_i}}{2\Z_n}\left(-\Pf(K_n^{00})_{E^c}+\Pf(K_n^{10})_{E^c}+\Pf(K_n^{01})_{E^c}+\Pf(K_n^{11})_{E^c}\right),
\end{equation}
where $E^c=V(\GD_n)\setminus\{u_1,v_1,\cdots,u_m,v_m\}$, and
$(K_n^{\theta\tau})_{E^c}$ is the submatrix of $K_n^{\theta\tau}$ whose lines
and columns are indexed by $E^c$.
\end{enumerate}
\end{thm}

\subsection{Critical dimer model on infinite periodic Fisher graphs}
In this section, we start by defining the Kasteleyn operator of the infinite $\ZZ^2$-periodic Fisher graph $\GD$, then Proposition \ref{prop:uniqueness} shows existence and uniqueness of its inverse. Theorem \ref{prop:periodic} proves an explicit expression for the free energy, and for the Gibbs measure obtained as weak limit of Boltzmann measures for the critical dimer model on $\GD$.

\subsubsection{Kasteleyn operator and characteristic polynomial}
\label{sec:Kfourier}

Assume that the infinite graph $\GD$ has the periodic Kasteleyn orientation induced by the Kasteleyn orientation of $\GD_1$. Let $K$ be the corresponding infinite Kasteleyn matrix. Then $K$ is an operator acting on $\CC^{V(\GD)}$:
\begin{equation*}
\forall f\in \CC^{V(\GD)},\quad (Kf)_u=\sum_{u'\in V(\GD)}K_{u,u'}f_{u'}.
\end{equation*}
The operator $K$ acts by convolution when interpreted as follows. Since $\GD$ is $\ZZ^2$-periodic, vertices of $\GD$ can be identified with $V(\GD_1)\times \ZZ^2$, {\em i.e.} $(v,x,y)$ is the vertex $v$ in the $(x,y)$ copy of the fundamental domain $\GD_1$. Coefficients of the infinite Kasteleyn matrix $K$ are thus indexed as $K_{(v,x,y),(v',x',y')}$. A function $f\in\CC^{V(\GD)}$ can be written as a complex vector-valued function $f:\ZZ^2\rightarrow\CC^{V(\GD_1)}$,
\begin{equation*}
 (f(x,y))_{v}=f(v,x,y).
\end{equation*}
Since $K$ is translation invariant, it acts as:
\begin{align*}
((K f)(x,y))_{v}&=\sum_{(v',x',y')}K_{(v,x,y),(v',x',y')}(f(x',y'))_{v'}
=\sum_{(v',x',y')}K_{(v,x-x',y-y'),(v',0,0)}(f(x',y'))_{v'}\\
&=((K\ast f)(x,y))_{v},
\end{align*}
where $K:\ZZ^2\rightarrow M(|V(\GD_1)|\times|V(\GD_1)|)$ is the complex matrix-valued function defined~by:
\begin{equation*}
K(x,y)_{v,v'}=K_{(v,x,y),(v',0,0)}.
\end{equation*}
Introducing a new notation for the matrix-valued function defined above seemed artificial to us, we nevertheless hope that this will not confuse the reader.

Let $\widehat{K}:\TT^2\rightarrow M(|V(\GD_1)|\times|V(\GD_1)|)$ be the Fourier transform of $K$:
\begin{equation*}
\widehat{K}(z,w)=\sum_{(x,y)\in\ZZ^2}K(x,y)z^{x} w^{y}.
\end{equation*}
The {\em dimer characteristic polynomial} of the graph $\GD$, denoted by $P(z,w)$, is defined to be the determinant of $\widehat{K}(z,w)$. The set
\begin{equation*}
 \{(z,w)\in\CC^2\,:\,P(z,w)=\det(\widehat{K}(z,w))=0\},
\end{equation*}
is the {\em dimer spectral curve} of the graph $\GD$. 

Theorem \ref{thm:thm2}, see also Theorem \ref{thm:P_Harnack} of Section \ref{sec:P_Harnack} states that the dimer spectral curve of the graph $\GD$ is a Harnack curve of genus $0$, and that the dimer characteristic polynomial of the graph $\GD$ has a unique zero on $\TT^2$, located at $(1,1)$, with multiplicity two. The proof of Theorem \ref{thm:P_Harnack} is postponed until Section \ref{sec:P_Harnack}. These properties of the characteristic polynomial of $\GD$ are used in a crucial way in the next two sections.

\subsubsection{Inverse of the Kasteleyn operator}
\label{sec:Kinverse}

Let us define the space of rapidly decaying vector-valued functions $\mathcal{S}(\ZZ^2)$, and the space of vector-valued functions decaying at infinity $\mathcal{C}_0(\ZZ^2)$:
\begin{align*}
  \mathcal{S}(\ZZ^2)&=\{ f:\ZZ^2\rightarrow\CC^{V(\GD_1)} \,:\, \forall (m,n)\in\NN^2, \lim_{\|(x,y)\|\rightarrow\infty}\| x^m y^n f(x,y)\| =0\},\\
  \mathcal{C}_0(\ZZ^2)&=\{ f:\ZZ^2\rightarrow\CC^{V(\GD_1)} \,:\, \forall \varepsilon >0, \exists N\in \NN, \sup_{\|(x,y)\|\geq N} \|f(x,y)\| \leq \varepsilon\}.
\end{align*}

Proposition \ref{prop:uniqueness} below states the existence and uniqueness of the inverse Kasteleyn operator. The proof is inspired from the classical one for the continuous Laplacian on $\RR^n$, using Fourier transform.
\begin{prop}\label{prop:uniqueness}
For every function $f\in\mathcal{S}(\ZZ^2)$, there exists a unique function $u\in\mathcal{C}_0(\ZZ^2)$, such that $K u =f$.
\end{prop}
\begin{proof}[Proof of existence]
  By Theorem \ref{thm:thm2}, for $(z,w)\in\TT^2\setminus\{(1,1)\}$, $P(z,w)=\det(\widehat{K}(z,w))~\neq~ 0$. As a consequence, the following matrix valued function is well defined on $\TT^2\setminus\{(1,1)\}$:
\begin{equation*}
  \widehat{K}^{-1}(z,w)=\left(\widehat{K}(z,w)\right)^{-1} =\frac{\mathrm{Cof} (\widehat{K}(z,w))^t}{P(z,w)},
\end{equation*}
where $\mathrm{Cof} (\widehat{K}(z,w))$ is the cofactor matrix of the matrix $\widehat{K}(z,w)$.
Moreover, since $\widehat{K}(1,1)$ is skew-symmetric of even size and not invertible, the dimension of its kernel is even and non zero, thus at least equal to $2$. Therefore $\mathrm{Cof}(\widehat{K}(1,1))$ is identically $0$, and the singularity of $\widehat{K}^{-1}(z,w)$ at $(1,1)$ is of order $1$ at most, and not of order $2$ as one could have expected. This implies that $\widehat{K}^{-1}$
is integrable on $\TT^2$. Its inverse Fourier transform is thus well defined, let us denote it by $U$:
\begin{equation*}
  U(x,y)=\F^{-1}(\widehat{K}^{-1})(x,y)=\frac{1}{(2\pi i)^2}\iint_{\TT^2}\frac{\mathrm{Cof}( \widehat{K}(z,w))^t}{P(z,w)}z^{-x} w^{-y}\frac{\ud z}{z}\frac{\ud w}{w}.
\end{equation*}
Define the complex vector-valued function $u:\ZZ^2\rightarrow\CC^{V(\GD_1)}$ by:
\begin{equation*}
u=U \ast f.
\end{equation*}
Let us first prove that $K u=K\ast u=f$. For every $(x,y),\,(x_2,y_2)\in\ZZ^2$, we have:
\begin{multline*} 
\sum_{(x_1,y_1)\in\ZZ^2}K(x-x_1,y-y_1)U(x_1-x_2,y_1-y_2)
= \\
\frac{1}{(2\pi i)^2} 
\iint_{\TT^2}\Id_{|V(\GD_1)|}z^{x-x_2}w^{y-y_2}\frac{\ud z}{ z}\frac{\ud w}{ w}
=\Id_{|V(\GD_1)|}\delta_{x,x_2}\delta_{y,y_2}.
\end{multline*}
Since the sum over $(x_1,y_1)$ is in fact finite, we deduce:
\begin{align*}
(K\ast U\ast f)(x,y)&=\sum_{(x_1,y_1)\in\ZZ^2}\sum_{(x_2,y_2)\in\ZZ^2}
K(x-x_1,y-y_1)U(x_1-x_2,y_1-y_2)f(x_2,y_2)\\
&=\sum_{(x_2,y_2)\in\ZZ^2}\delta_{x,x_2}\delta_{y,y_2}\Id_{|V(\GD_1)|}f(x_2,y_2)=f(x,y).
\end{align*}
Let us now prove that $u$ is in $\mathcal{C}_0(\ZZ^2)$. Split the sum over $(x',y')$ into two parts, depending on the distance to $(x,y)$: 
\begin{equation*}
u(x,y)=\sum_{
\begin{subarray}{c}
(x',y')\\
\ud((x,y),(x',y'))>L
\end{subarray}
}U(x-x',y-y')f(x',y')+
\sum_{
\begin{subarray}{c}
(x',y')\\
\ud((x,y),(x',y'))\leq L
\end{subarray}}
U(x-x',y-y')f(x',y').
\end{equation*}
By classical estimate on the Fourier coefficients of a rational fraction (see for example \cite{KOS}, Lemma 4.4), we know that $U(x-x',y-y')$ decays linearly with the distance $\mathrm{d}((x,y),(x',y'))$. We can therefore choose $L$ large enough, so that  
\begin{equation*}
  \|U(x-x',y-y')\|\leq \frac{\varepsilon}{2 \|f\|_1 +1},
\end{equation*}
when $\mathrm{d}((x,y),(x',y'))~>~L$. The first sum is thus bounded by:
\begin{equation*}
 \sum_{
\begin{subarray}{c}
(x',y')\\
\ud((x,y),(x',y'))>L
\end{subarray}}
\left\|
U(x-x',y-y')\|\,\|f(x',y')
\right\| \leq 
\frac{\varepsilon}{2 ||f||_1 +1}\sum_{(x',y')}\|f(x',y')\| \leq \frac{\varepsilon}{2}.
 \end{equation*}
Since the coefficients $U(x-x',y-y')$ are bounded and $f\in\mathcal{S}(\ZZ^2)$, one can choose $N=N(L)$ large enough so that if $\|(x,y)\|\geq N$, then for all $(x',y')$ such that $\ud((x,y),(x',y'))\leq L$, we have $\|f(x',y')\|\leq \frac{\varepsilon}{2  n_L ||U||_{\infty}}$, where $n_L$ is the number of points in the ball of radius $L$.
In that case, the second sum is thus bounded by:
\begin{equation*}
\sum_{
\begin{subarray}{c}
(x',y')\\
\ud((x,y),(x',y'))\leq L
\end{subarray}}
\|U(x-x',y-y')\|\,\|f(x',y')\| \leq n_L
\|U\|_{\infty}\sup_{
\begin{subarray}{c}
(x',y')\\
\ud((x,y),(x',y'))\leq L
\end{subarray}}
\|f(x',y')\|
\leq \frac{\varepsilon}{2}.
\end{equation*}

Thus, when $\|(x,y)\|\geq N$, $\|u(x,y)\|\leq \varepsilon$, and we conclude that $u\in\mathcal{C}_{0}(\ZZ^2)$.

{\em Proof of uniqueness.}
Define $\mathcal{S}(\TT^2)$ to be the space of smooth functions on the torus with values in $\CC^{V(G_1)}$. Recall that the Fourier transform $\mathcal{F}$ is a bijection between $\mathcal{S}(\ZZ^2)$ and $\mathcal{S}(\TT^2)$.
Let us denote by $\langle\ ,\ \rangle$ the duality bracket between $\mathcal{S}(\ZZ^2)$ and its dual  $\mathcal{S'}(\ZZ^2)$, and by $(\ ,\ )$ the one between $\mathcal{S}(\TT^2)$ and its dual  $\mathcal{S'}(\TT^2)$.

By duality, the Fourier transform extends as a bijection from $\mathcal{S'}(\ZZ^2)$ to $ \mathcal{S'}(\TT^2)$:
\begin{equation*}
\forall u\in\mathcal{S}'(\ZZ^2),\;\;\forall \psi\in\mathcal{S}(\TT^2),\quad\quad
(\F u,\psi)=\langle u,\F^{-1}\psi\rangle.
\end{equation*}

Suppose that we have two functions $u_1$ and $u_2$ in $\mathcal{C}_{0}(\ZZ^2)$, such that $K\ast u_1=K\ast u_2 =f$. Let $u=u_2-u_1$. 
The function $u\in \mathcal{C}_{0}(\ZZ^2)$, defines a distribution $u$ in $\mathcal{S'}(\ZZ^2)$. 

Then, since $K\ast u=0$, we have for every $\psi\in \mathcal{S}(\TT^2)$,
\begin{multline}\label{eq:fourier}
0=(\F(K\ast u),\psi)=\langle K\ast u,\F^{-1}\psi\rangle =-\langle u,K\ast \F^{-1}\psi\rangle  \\
	   =-\langle u,\F^{-1}(\widehat{K}\psi)\rangle =-(\F(u),\widehat{K}\psi).
								  \end{multline}
 The third equality results from the fact that $K$ is a skew-symmetric operator.

Let $\xi$ be a function in $\mathcal{S}(\TT^2)$ whose support is contained in $\TT^2\setminus \{(1,1)\}$. Since $\widehat{K}^{-1}$, as a function of $(z,w)$ is smooth on $\TT^2\setminus \{(1,1)\}$, we deduce that 
\begin{equation*}
\psi=\widehat{K}^{-1}\xi\in \mathcal{S}(\TT^2).
\end{equation*}
Plugging this in \eqref{eq:fourier}, we deduce:
\begin{equation*}
\forall \xi \in \mathcal{S}(\TT^2)\text{ with supp}(\xi)\subset\TT^2\setminus\{(1,1)\},\quad(\F(u),\xi)=0.
\end{equation*}
This implies that the support of $\mathcal{F}(u)$ is contained in $\{(1,1)\}$. As a consequence, $\F(u)$ is a linear combination of derivatives of the Dirac distribution at $(1,1)$. But the only, such linear combination which has an inverse Fourier transform in $\mathcal{C}_0(\ZZ^2)$ is the trivial one identically equal to $0$, therefore, $\F(u)=0$, {\em i.e.} $u=0$.
\end{proof}

In the light of Proposition \ref{prop:uniqueness}, let us define the {\em inverse Kasteleyn matrix}, denoted by $K^{-1}$, to be the infinite matrix whose coefficients are:
\begin{equation}\label{eq:Kinverse}
K^{-1}_{(v,x,y),(v',x',y')}=\frac{1}{(2\pi i)^2}\iint_{\TT^2}\frac{\text{Cof}(\widehat{K}(z,w))_{v',v}}{P(z,w)}z^{x'-x}w^{y'-y}\frac{\ud z}{z}\frac{\ud w}{w}.
\end{equation}

\subsubsection{Free energy and Gibbs measure}

The {\em free energy} of the dimer model on the infinite $\ZZ^2$-periodic graph $\GD$ is denoted $f$, and is defined by:
\begin{equation*}
  f=-\lim_{n\rightarrow\infty}\frac{1}{n^2}\log \mathcal{Z}_n.
\end{equation*}

A {\em Gibbs measure} on the set of dimer configurations $\M(\GD)$ of $\GD$, is
a probability measure on $\M(\GD)$, which satisfies the following. If one
fixes a perfect matching in an annular region of $\GD$, then perfect matchings
inside and outside of this annulus are independent. Moreover, the probability
of occurrence of an interior matching is proportional to the product of its
edges weights.

In order to state Theorem \ref{prop:periodic} giving an explicit expression for the free energy of the dimer model on $\GD$, and for the Gibbs measure obtained as weak limit of the Boltzmann measures, we need the following notation: $\mathfrak{F}$ is the $\sigma$-field generated by cylinders, a {\em cylinder} being the set of dimer configurations of $\GD$ containing a fixed finite subset of edges of $\GD$.

\begin{thm}\label{prop:periodic}$\,$
\begin{enumerate}
\item The free energy of the dimer model on $\GD$ is: 
\begin{equation*}
f=-\frac{1}{2(2\pi i)^2}\iint_{\TT^2}\log P(z,w)\frac{\ud z}{z}\frac{\ud w}{w}.
\end{equation*}
\item There is a unique probability measure $\P$ on $(\M(\GD),\mathfrak{F})$, such
that the probability of occurrence of a subset of edges $E=\{e_1=u_1 v_1,\cdots,e_m=u_m v_m\}$, of $\GD$ in a dimer configuration of $\GD$ is:
\begin{equation}\label{eq:periodic}
\P(e_1,\cdots,e_m)=\left(\prod_{i=1}^m K_{u_i,v_i}\right)\Pf((K^{-1})_E),
\end{equation}
where $(K^{-1})_E$ is the submatrix of the infinite matrix $K^{-1}$ given in \eqref{eq:Kinverse}, whose lines and columns are indexed by the vertices defining the edges of $E$. Moreover, $\P$ is a translation invariant ergodic Gibbs measure.
\end{enumerate}
\end{thm} 

\begin{proof}
By Theorem \ref{thm:thm2}, we know that $P(z,w)$ has a unique double zero at $(1,1)$ on the unit torus. With this information, the general argument of \cite{CKP} works for Part $1$. We nevertheless sketch the proof for the convenience of the reader.

For the proof of Parts $1$ and $2$ of Theorem \ref{prop:periodic}, we need the following facts.
\begin{enumerate}
\item[(a)] This argument can be found in \cite{CKP}.  For every
$\theta,\tau\in\{0,1\}$, using Fourier transform, the matrix
$K_n^{\theta\tau}$ can be block diagonalized, with $(n^2)^2$ blocks of size
$|V(\GD_1)|\times|V(\GD_1)|$:
\begin{equation}\label{eq:fourier2}
K_n^{\theta\tau}=(Q_n^{\theta\tau})^\dagger B_n^{\theta\tau}Q_n^{\theta\tau},
\end{equation}  
where the matrices $Q_n^{\theta\tau}$ and $B_n^{\theta\tau}$ are given in
block form ($x,y,j,k,j',k'\in\{~0,~\cdots~,~n-1~\}$):
\begin{align*}
(Q_n^{\theta\tau})_{(\cdot,j,k),(\cdot,x,y)}&=
\frac{1}{n}e^{i\frac{(2j+\theta)\pi}{n}x}
e^{i\frac{(2k+\tau)\pi}{n}y}
\Id_{|V(\GD_1)|},\\
(B_n^{\theta\tau})_{(\cdot,j,k),(\cdot,j',k')}&=\delta_{j,j'}\delta_{k,k'}\widehat{K}\left(e^{i\frac{(2j+\theta)\pi}{n}},e^{i\frac{(2k+\tau)\pi}{n}}\right).
\end{align*}
Since the matrix $B_n^{\theta\tau}$ is block diagonal, we simply write
$(B_n^{\theta\tau})_{(\cdot,j,k)}$ for the $(\cdot,j,k)$ diagonal block.
The matrix $Q_n^{\theta\tau}$ is unitary, i.e. $(Q_n^{\theta\tau})^\dagger
Q_n^{\theta\tau}=\Id_{(|V(\GD_1)|n^2)}$.
\item[(b)] Recall that $P(z,w)=\det(\widehat{K}(z,w))$, and that by Theorem \ref{thm:thm2}, $P(z,w)$ has a single double zero at $(1,1)$ on the unit torus. We will use the following facts:
\begin{enumerate}
\item[$\circ$] $\left(e^{i\frac{(2j+\theta)\pi}{n}},e^{i\frac{(2k+\tau)\pi}{n}}\right)=(1,1)\Leftrightarrow
  (j,k)=(0,0)\text{ and }(\theta,\tau)=(0,0)$.
\item[$\circ$] For all other indices, the points
  $\left(e^{i\frac{(2j+\theta)\pi}{n}},e^{i\frac{(2k+\tau)\pi}{n}}\right)$ are at distance at least
  $O\left(\frac{1}{n}\right)$ from the zero $(1,1)$ of $P$ on the torus.
\end{enumerate}
\end{enumerate}

\underline{\em Proof of Part 1}.

As a consequence of Equation \eqref{eq:fourier2}:
\begin{equation}\label{eq:detK_prod}
\det{K_n^{\theta\tau}}=\det{B_n^{\theta\tau}}=\prod_{j=0}^{n-1}\prod_{k=0}^{n-1}P\left(e^{i\frac{(2j+\theta)\pi}{n}},e^{i\frac{(2k+\tau)\pi}{n}}\right).
\end{equation}
Using (b), we deduce that for every $n$, $\det(K_n^{\theta\tau})=0$, if and only if
$(\theta,\tau)=(0,0)$. Moreover, for $(\theta,\tau)\neq (0,0)$, the points 
 $(e^{i\frac{(2j+\theta)\pi}{n}},e^{i\frac{(2k+\tau)\pi}{n}})$ are at distance at least
  $O\left(\frac{1}{n}\right)$ from the zero of $P$, so that by the general argument of
  \cite{CKP} (see also \cite{KOS}), we have:
\begin{align*}
\lim_{n\rightarrow\infty}\frac{1}{n^2}\log
\Pf(K_n^{\theta\tau})&=\frac{1}{2}\lim_{n\rightarrow\infty}\frac{1}{n^2}\log\det(K_n^{\theta\tau})\\
&=\frac{1}{2}\frac{1}{(2\pi)^2}\lim_{n\rightarrow\infty}\frac{(2\pi)^2}{n^2}\sum_{j=0}^{n-1}\sum_{k=0}^{n-1}\log
P\left(e^{i\frac{(2j+\theta)\pi}{n}},e^{i\frac{(2k+\tau)\pi}{n}}\right),\\
&=\frac{1}{2}\frac{1}{(2\pi)^2}\int_{0}^{2\pi}\int_{0}^{2\pi}\log
P(e^{i\eta},e^{i\xi})\,d\eta\,d\xi.
\end{align*}
By Theorem \ref{thm:torus}, 
\begin{equation}
  \max_{\theta,\tau}\{\Pf(K_n^{\theta\tau})\}\leq \Z_n\leq
2\max_{\theta,\tau}\{\Pf(K_n^{\theta\tau})\},
\label{eq:encadrementZ}
\end{equation}
and we deduce Part $1$ of Theorem \ref{prop:periodic}.

\underline{\em Proof of Part $2$}.

Let us prove that the Boltzmann probability $\P_n(e_1,\cdots,e_m)$ given in \eqref{eq:bolztorus} of the edges $e_1=u_1 v_1,\cdots,e_m=u_m v_m$ being in a dimer configuration of $\GD_n$ ($n$ large enough) converges to the RHS of
\eqref{eq:periodic}, as $n\rightarrow\infty$. The existence of a unique
measure $\P$ equal to the RHS of \eqref{eq:periodic} on cylinder sets is then
given by Kolomogorov's extension theorem (a similar argument is done in
details in \cite{Bea1}). The Gibbs property results from the fact that $\mathcal{P}$ is a weak limit of Boltzmann measures. Translation invariance is a consequence of the periodicity of the weights, and ergodicity (even mixing) readily follows from the decay of the coefficients of $K^{-1}$ with the distance between vertices.

\underline{\em Case 1:} {\em Terms of \eqref{eq:bolztorus} involving
  $\Pf(K_n^{\theta\tau})_{E^c}$, $(\theta,\tau)\neq (0,0)$.}
In this case, the general argument of \cite{CKP} works (see
also \cite{KOS}). It runs as follows. By (b), for every $n$,
$\det(K_n^{\theta\tau})\neq 0$. 
We thus apply Jacobi's
formula for the determinant of minor matrices. Observing that all matrices
involved are skew symmetric, this yields:
\begin{equation*}
\Pf(K_n^{\theta\tau})_{E^c}=\Pf((K_n^{\theta\tau})^{-1})_E\Pf(K_n^{\theta\tau}).
\end{equation*}
Moreover, by (a), we have for all $(\theta,\tau)\neq (0,0)$:
\begin{multline*}
(K_n^{\theta\tau})^{-1}_{(v,x,y),(v',x',y')}=\\
=\frac{1}{n^2}\sum_{j=0}^{n-1}\sum_{k=0}^{n-1}e^{i\frac{(2j+\theta)\pi}{n}(x' -x)}e^{i\frac{(2k+\tau)\pi}{n}(y'-y)}
\frac{\Cof\left( \widehat{K}\left(e^{i\frac{(2j+\theta)\pi}{n}},e^{i\frac{(2k+\tau)\pi}{n}}\right)\right)_{v',v}}{P\left(e^{i\frac{(2j+\theta)\pi}{n}},e^{i\frac{(2k+\tau)\pi}{n}}\right)}.
\end{multline*}
Since the points $\left(e^{i\frac{(2j+\theta)\pi}{n}},e^{i\frac{(2k+\tau)\pi}{n}}\right)$ are at distance at least
  $O\left(\frac{1}{n}\right)$ from the zero of $P$, the contribution of the points close to the singularity is negligible, and the Riemann sum converges
  to $K^{-1}_{(v,x,y),(v',x',y')}$ given in \eqref{eq:Kinverse}. 

Moreover, $\frac{\Pf(K_n^{\theta\tau})}{2\Z_n}$
is bounded, and since $\Pf(K_n^{00})=0$,
$
\sum_{(\theta,\tau)\neq(0,0)}\frac{\Pf(K_n^{\theta\tau})}{2\Z_n}=1.
$
We conclude that:
\begin{equation*}
\lim_{n\rightarrow\infty}\frac{\Pf(K_n^{10})_{E^c}+\Pf(K_n^{01})_{E^c}+\Pf(K_n^{11})_{E^c}}{2\Z_n}=
\Pf(K^{-1})_E,
\end{equation*}
where $(K^{-1})_E$ is the submatrix of the infinite matrix $K^{-1}$, whose lines and columns are indexed by vertices of $E$.

\underline{\em Case 2:} {\em Term of \eqref{eq:bolztorus} involving
  $\Pf(K_n^{00})_{E^c}$.}
Let us show that, $\lim_{n\rightarrow\infty} \frac{\Pf(K_n^{00})_{E^c}}{2\Z_n}=0$. We actually prove the following precise estimate:
\begin{equation}\label{eq:firstterm}
  \frac{\Pf(K_n^{00})_{E^c}}{\Z_n}= O\left( \frac{1}{n} \right).
\end{equation}

To simplify notations, we write $K_n^{00}=K_n$, $Q_n^{00}=Q$,
$B_n^{00}=B$.

By (a), $B$ is skew-Hermitian, so that the diagonal block $B_{(\cdot,0,0)}$
is too. Moreover by (b), $B_{(\cdot,0,0)}=\widehat{K}(1,1)$ is not invertible since
$\det(\widehat{K}(1,1))=P(1,1)=0$. Let $2q$ be the dimension of the kernel of
$B_{(\cdot,0,0)}$, and $P_0$ be an orthogonal matrix of size $|V(\GD_1)|\times|V(\GD_1)|$, such that:
\begin{equation*}
B_{(\cdot,0,0)}=P_0^t D_0 P_0,
\end{equation*}
where $D_0$ is diagonal, with the first $2q$ elements being $0$. Observe that
$P_0$ and $D_0$ are independent of $n$.

Define $\Qs$ to be the matrix obtained from $Q$ by multiplying all blocks of
the first block line by $P$ on the left. Then, $\Qs$ is unitary and
\begin{equation*}
K_n=\Qs^\dagger \Bs \Qs,
\end{equation*}
where $\Bs$ is block diagonal, with:
\begin{align*}
\Bs_{(\cdot,0,0)}&=D_0,\\
\Bs_{(\cdot,j,k)}&=B_{(\cdot,j,k)}, \text{ for all
}(j,k)\neq(0,0).
\end{align*}
Let $\bar\Bs$ be the matrix obtained from $\Bs$ by removing
the first $2q$ lines and first $2q$ columns. Since, for $(j,k)\neq (0,0)$, all blocks
$B_{(\cdot,j,k)}$ are invertible, the matrix $\bar\Bs$ is.

The proof of \eqref{eq:firstterm} is a consequence of:
\begin{lem} \label{lem:firstterm1} We have the following estimates:
\begin{enumerate}
  \item $\displaystyle \frac{\Pf(K_n^{00})_{E^c}}{\sqrt{\det\bar\Bs}}=O\left( \frac{1}{n^{2q}} \right).$
  \item $\displaystyle \frac{\sqrt{\det \bar\Bs}}{\Z_n}=O\left( n \right).$
\end{enumerate}
\end{lem}

Before proving Lemma \ref{lem:firstterm1}, let us end the proof of Theorem \ref{prop:periodic}.
Combining the above two estimates, we see that the LHS of \eqref{eq:firstterm} is $O(n^{1-2q})$, and since $q\geq 1$, it goes to zero at least as fast as $n^{-1}$, as $n$ goes to infinity.
\end{proof}

\begin{proof}[Proof of Part $1$ of Lemma \ref{lem:firstterm1}]
  Recall that $|E|=2m$. To simplify notations, let us set $2N:=|V(\GD_1)|n^2$. We will use the following convention: if $A$ is a matrix, and $S$ (resp. $T$) is a subset of row-indices (resp. column-indices) of $A$, the $A_S^T$ is the submatrix obtained from $A$ by keeping rows with indices in $S$ and columns with indices in $T$. If $S$ (resp. $T$) is omitted, then all the rows (resp. columns) are kept\footnote{To comply with the notations used before and to simply notations, there is one exception to this rule: $(K_n)_{E^c}$ still denotes the submatrix of $K_n$ with rows \emph{and} columns indexed by $E^c$.}.

We apply Cauchy-Binet's formula to compute $\det(K_n)_{E^c}$:
\begin{align*}
\det(K_n)_{E^c}&=\det((\Qs^\dagger\Bs\Qs)_{E^c}^{E^c})=
\det(\Qs_{E^c}^\dagger\Bs\Qs^{E^c})\\
&=\sum_{\begin{subarray}{l} 
S,\,T\,\subset\{1,\cdots,2N\}\\
|S|=|T|=2N-2m
\end{subarray}}
\det(\Qs_{E^c}^{\dagger\,S})\det(\Bs_S^T)\det(\Qs_{T}^{E^c}).
\end{align*}

Since the first $2q$ lines and $2q$ columns of $\Bs$ are $0$, $\det(K_n)=0$ if
$m<q$, so that \eqref{eq:firstterm} is true. Let us assume that $m\geq
q$. Then,
\begin{equation*}
\det(K_n)_{E^c}=
\sum_{\begin{subarray}{l} 
S,\,T\,\subset\{2q+1,\cdots,2N\}\\
|S|=|T|=2N-2m
\end{subarray}}
\det(\Qs_{E^c}^{\dagger\,S})\det(\Bs_S^T)\det(\Qs_{T}^{E^c}).
\end{equation*}
Let $\underline{S}^c=\{2q+1,\cdots,2N\}\setminus
S,\,\underline{T}^c=\{2q+1,\cdots,2N\}\setminus T$. Then,
by Jabobi's formula and unitarity of $\Qs$, we have:
\begin{gather*}
\det(\Bs_S^T)=\det(\bar\Bs)\det((\bar\Bs^{-1})_{\underline{T}^c}^{\underline{S}^c}),\\
\det(\Qs_{T}^{E^c})=\det(\Qs)\det(\Qs_{E}^{\dagger\,T^c}),\quad
\det(\Qs_{E^c}^{\dagger\,S})=\det(\Qs^\dagger)\det(\Qs_{S^c}^{E}).
\end{gather*}
So that:
\begin{equation*}
\det((K_n)_{E^c})=\det(\bar\Bs)
\sum_{\begin{subarray}{l} 
S,\,T\,\subset\{2q+1,\cdots,2N\}\\
|S|=|T|=2N-2m
\end{subarray}}
\det(\Qs_{E}^{\dagger\,T^c})\det((\bar\Bs^{-1})_{\underline{T}^c}^{\underline{S}^c})\det(\Qs_{S^c}^{E}).
\end{equation*}
We have the following expansions for the determinants, see for example
\cite{Bourbaki:alg}, A~III.98, formula $(21)$:
\begin{align*}
\det(\Qs_{S^c}^{E})&=\sum_{
\begin{subarray}{l} 
L\,\subset\, E\\
|L|=2q
\end{subarray}}
\rho_{L,L'}\det(\Qs_{\{1,\cdots,2q\}}^L)\det(\Qs_{\underline{S}^c}^{L'}),\\
\det(\Qs_{E}^{\dagger\,T^c})&=\sum_{
\begin{subarray}{l} 
K\,\subset\, E\\
|K|=2q
\end{subarray}}
\rho_{K,K'}\det(\Qs_{K}^{\dagger\,\{1,\cdots,2q\}})\det(\Qs_{K'}^{\dagger\,\underline{T}^c}),
\end{align*}
where $L'=S^c\setminus L$, $K'=T^c\setminus K$, and $\rho_{L,L'}$,
$\rho_{K,K'}$ are signature of permutations constructed from $L,\,L'$ and $K,\,K'$, see
\cite{Bourbaki:alg} for a precise definition.

Thus,
\begin{multline*}
\frac{\det(K_n)_{E^c}}{\det(\bar\Bs)}=\sum_{
\begin{subarray}{l} 
K,\,L\,\subset\, E\\
|K|=|L|=2q
\end{subarray}}
\rho_{K,K'}\,\rho_{L,L'}\det(\Qs_{K}^{\dagger\,\{1,\cdots,2q\}})\det(\Qs_{\{1,\cdots,2q\}}^L)\\
\sum_{
\begin{subarray}{l} 
\underline{S}^c\,\underline{T}^c\,\subset\, \{2q+1,\cdots,2N\}\\
|\underline{S}^c|=|\underline{T}^c|=2m-2q
\end{subarray}}
\det(\Qs_{K'}^{\dagger\,\underline{T}^c})\det((\bar\Bs^{-1})_{\underline{T}^c}^{\underline{S}^c})\det(\Qs_{\underline{S}^c}^{L'}).
\end{multline*}
Let $\widetilde{\Qs}$ be the matrix of size $(2N-2q)\times 2N$ obtained form $\Qs$
by removing the first $2q$ lines. Then, using Cauchy-Binet's formula again, we have:
\begin{equation*}
\frac{\det(K_n)_{E^c}}{\det(\bar\Bs)}=\sum_{
\begin{subarray}{l} 
K,\,L\,\subset\, E\\
|K|=|L|=2q
\end{subarray}}
\rho_{K,K'}\,\rho_{L,L'}\det(\Qs_{K}^{\dagger\,\{1,\cdots,2q\}})\det(\Qs_{\{1,\cdots,2q\}}^L)
\det((\widetilde{\Qs}^\dagger\bar\Bs^{-1}\widetilde{\Qs})_{K'}^{L'}).
\end{equation*}
Now, by definition of the matrix $\Qs$, all elements of
$\Qs_{K}^{\dagger\,\{1,\cdots,2q\}}$ and $\Qs_{\{1,\cdots,2q\}}^L$ are $O(n^{-1})$. Hence, by Hadamard's inequality,
$\det(\Qs_{K}^{\dagger\,\{1,\cdots,2q\}})$  and $\det(\Qs_{\{1,\cdots,2q\}}^L)$
 are $O\left(\frac{1}{n^{2q}}\right)$.
Moreover, by definition of the matrix $\widetilde{\Qs}$ and $\bar\Bs$, 
\begin{equation*}
(\widetilde{\Qs}^\dagger\bar\Bs^{-1}\widetilde{\Qs})_{(v,x,y)(v',x',y')}=
\frac{c_{vv'}}{n^2}+
\frac{1}{n^2} \sum_{j=0}^{n-1} \sum_{\substack{k=0\\ (j,k)\neq(0,0)}}^{n-1} e^{i\frac{2j\pi}{n}(x'-x)} e^{i\frac{2k\pi}{n}(y'-y)}
\frac{\Cof\left( \widehat{K}\left(e^{i\frac{2j\pi}{n}},e^{i\frac{2k\pi}{n}}\right)\right)_{v',v}}{P\left(e^{i\frac{2j\pi}{n}},e^{i\frac{2k\pi}{n}}\right)},
\end{equation*}
for some number $c_{vv'}$ independent of $n$, coming from the contribution of the first block of $\bar\Bs$. Since $(j,k)\neq(0,0)$, all points
$\left(e^{i\frac{2j\pi}{n}},e^{i\frac{2k\pi}{n}}\right)$ are at distance at least
$O\left(\frac{1}{n}\right)$ from the $0$ of $P$, and the Riemann sum converges
to $K^{-1}_{(v,x,y)(v',x',y')}$ given in \eqref{eq:Kinverse}. Moreover, these terms
are uniformly bounded, so that:
\begin{equation*}
\left\lvert \det(\widetilde{\Qs}^\dagger\bar\Bs^{-1}\widetilde{\Qs})_{K'}^{L'}\right\rvert\leq C,
\end{equation*}
for some constant $C$. Since $k$ and $q$ are independent of $n$, we deduce
that:
\begin{equation*}
\frac{\det((K_n)_{E^c})}{\det(\bar\Bs)}=O\left(\frac{1}{n^{4q}}\right).
\end{equation*}
\end{proof}

\begin{proof}[Proof of Part $2$ of Lemma \ref{lem:firstterm1}]
  Since the partition function $\Z_n$ is bounded from below by any of the $\mathrm{Pf}(K_n^{\theta\tau})$, it suffices to show that
  \begin{equation*}
    \det(\bar\Bs) =  \det(K_n^{10}) \times O(n^2).
  \end{equation*}

  By \eqref{eq:detK_prod} we have:
  \begin{align*}
    \left| \det(K_{n}^{10}) \right| &=\prod_{j,k\in\{-\lfloor \frac{n}{2}\rfloor,\dots,\lfloor \frac{n}{2}\rfloor-1\}} \left| P\left(e^{i\frac{(2j+1)\pi}{n}},e^{i\frac{(2k) \pi}{n}}\right)\right|,\\ 
    \left| \det(\bar\Bs)\right| &= \Lambda_0\prod_{\substack{j,k\in\{-\lfloor \frac{n}{2}\rfloor,\dots,\lfloor \frac{n}{2}\rfloor -1\} \\ (j,k)\neq (0,0)}} \left|  P\left(e^{i\frac{2j\pi}{n}},e^{i\frac{2k\pi}{n}}\right)\right|,
  \end{align*}

where $\Lambda_0$ is the product of the non-zero eigenvalues of $D_0$. 
Let us define:
   \begin{equation*}
    g(x,y) = P\left(e^{i\pi x},e^{i\pi y}\right), \quad  f(x,y) = \log \left|g(x,y)\right|= \Re \log(g(x,y)), 
  \end{equation*}
 then taking the logarithm of the ratio of the two determinants yields:
  \begin{equation}
    \log \left|\frac{\det K_n^{10}}{\det \bar\Bs}\right| = -\log\Lambda_0+\sum_{\substack{j,k\in\{-\lfloor \frac{n}{2}\rfloor,\dots,\lfloor \frac{n}{2}\rfloor-1\}\\ (j,k)\neq (0,0)}} \Delta_n^{(1)} f\left(\frac{2j}{n},\frac{2k}{n}\right) + \frac{1}{2} \sum_{\varepsilon \in \{-1,+1\}}f\left( \frac{\varepsilon}{n},0 \right),
    \label{eq:ratiodetBbar}
  \end{equation}
  where
  \begin{equation*}
    \Delta_n^{(1)} f(x,y) =  \frac{1}{2}\sum_{\varepsilon \in \{-1,+1\}} \left(f\left( x+\frac{\varepsilon}{n},y \right)-f(x,y)\right).
  \end{equation*}

The first term in \eqref{eq:ratiodetBbar}, $-\log\Lambda_0$ is a constant. The last sum  is equivalent to $2 \log n$ because, again by Theorem \ref{thm:P_Harnack}, $(1,1)$ is a  zero of $P$ of order exactly 2, and thus, for a certain $c\neq 0$, 
  \begin{equation*}
    f\left(\frac{\varepsilon}{n},0\right)=\log \left|\frac{c}{n^2} \left( 1+O\left( \frac{1}{n} \right) \right)\right| = 2 \log n + \text{bounded error terms}.
  \end{equation*}

To finish the proof, we need to show that first sum is $o(2\log n)$.
We first use Taylor expansion to show that for $(j,k)\neq(0,0)$,
  \begin{equation}
    \Delta_n^{(1)} f\left( \frac{2j}{n},\frac{2k}{n} \right) = \frac{1}{n^2} \frac{\partial^2 f}{\partial x^2} \left( \frac{2j}{n},\frac{2k}{n} \right)+ O\left( \frac{1}{n^{4}} \sup_{x\in\left[ \frac{2j-1}{n},\frac{2j+1}{n} \right] }\left|\frac{\partial^4 f}{\partial x^4}\left(  x,\frac{2k}{n} \right)\right|\right).
    \label{eq:Taylor_der2f}
  \end{equation}

  Let $V$ be a small neighborhood of $(0,0)$ in $[-1,1]^2$ contained in a ball of radius $\delta\in\bigr(0,\frac{1}{2}\bigr)$, and define 
  \begin{equation*}
    V_n= \left\{-\left\lfloor \frac{n}{2}\right\rfloor,\dots,\left\lfloor \frac{n}{2}\right\rfloor-1\right\}^2 \cap \frac{n}{2} V =\left\lbrace (j,k)\ : \left(\frac{2j}{n}, \frac{2k}{n}\right)\in V\right\rbrace.
  \end{equation*}
  As long as $(j,k)$ is outside of $V_n$, the fourth derivative of $f$ is uniformly bounded, and the discrete sum restricted to those indices $(j,k)$ converges to the integral of the smooth function $\frac{\partial^2 f}{\partial x ^2}$, on the square $[-1,1]^2$ deprived from $V$. Thus this part of the sum stays uniformly bounded as $n$ goes to infinity.

We now need to control the contribution to the first sum of \eqref{eq:ratiodetBbar} of the indices that are in $V_n$, \emph{i.e.} relatively close to zero. 
  Since $(1,1)$ is a double zero of $P$, we have for $x$ and $y$ close to zero the following expansions for $g$ and its derivatives:
  \begin{gather*}
    g(x,y)=P(e^{i\pi x},e^{i\pi y}) =\alpha x^2 + 2\beta x y +\gamma y^2 + O(\lVert(x,y)\rVert^3), \\
    \frac{\partial g}{\partial x}(x,y) = 2(\alpha x +\beta y) +O(\lVert(x,y)\rVert^2), \\
    \frac{\partial^2 g}{\partial x^2} (x,y)= 2\alpha + O\left( \lVert(x,y) \rVert \right),
  \end{gather*}
  with $\alpha,\beta,\gamma$ real numbers. Moreover, since $(1,1)$ is the only zero of $P$ on the unit torus, the quadratic form
  \begin{equation*}
  q(x,y)=\alpha x^2 + 2 \beta x y +\gamma y^2 = \alpha^{-1}\left(  (\alpha x + \beta y)^2 + ( \alpha \gamma-\beta^2)y^2 \right),
\end{equation*}
is definite. Without loss of generality, we can suppose that $\alpha$ is positive so that $q$ defines a scalar product on $\RR^2$.

  Let us now compute an expansion of the second derivative of $f$ with respect to the first variable:
\begin{align}
  \frac{\partial^2 f}{\partial x ^2}(x,y) &=\Re
  \frac{g(x,y)\frac{\partial^2 g}{\partial x^2}(x,y)-\left( \frac{\partial g}{\partial x}(x,y) \right)^2}{g(x,y)^2} \nonumber \\
  &= 2\alpha^2\left(\frac{-(\alpha x+\beta y)^2+(\alpha\gamma-\beta^2)y^2}{\left( (\alpha x + \beta y)^2 +(\alpha\gamma-\beta^2)y^2 \right)^2}\right) (1+O\left( ||(x,y)||\right))
  \label{eq:der2f},
\end{align}

if $(x,y)$ has the form $\left( \frac{2j}{n},\frac{2k}{n} \right)$, then \eqref{eq:der2f} divided by $n^{2}$ becomes:
\begin{equation}
  \frac{1}{n^2}  \frac{\partial^2 f}{\partial x^2}\left( \frac{2j}{n},\frac{2k}{n} \right)= \frac{\alpha^2}{2} \left( \frac{-(\alpha j + \beta k)^2 +(\alpha\gamma-\beta^2)k^2}{\left( (\alpha j +\beta k)^2 + (\alpha\gamma-\beta^2)k^2 \right)^2} \right) \left(1+O\left(\left\lVert\left(\frac{j}{n},\frac{k}{n}\right)\right\rVert\right)\right).
  \label{eq:mainterm_der2f}
\end{equation}

When we sum over $(j,k)$ in $V_n\setminus\{(0,0)\}$, the contribution of the error term is finite, since for some constant $C>0$,
\begin{equation*}
  \sum_{\substack{(j,k)\in U_n \\ (j,k)\neq(0,0)}} \left|\frac{-(\alpha j + \beta k)^2 +(\alpha\gamma-\beta^2)k^2}{\left( (\alpha j +\beta k)^2 + (\alpha\gamma-\beta^2)k^2 \right)^2}\times O\left(\left\lVert\left(\frac{j}{n},\frac{k}{n}\right)\right\rVert\right)\right|
  \leq \frac{C}{n} \sum_{\substack{\lVert(j,k)\lVert \leq \delta n \\ (j,k)\neq(0,0)}}\frac{1}{\left(   j^2+k^2 \right)^{1/2}},
\end{equation*}
which is bounded uniformly in $n$, as can be easily seen by estimating the sum using polar coordinates.

Let us now take care of the main contribution. We see that the expression

\begin{equation*}
  \frac{\alpha^2}{2} \left( \frac{-(\alpha j + \beta k)^2 +(\alpha\gamma-\beta^2)k^2}{\left( (\alpha j +\beta k)^2 + (\alpha\gamma-\beta^2)k^2 \right)^2} \right),
\end{equation*}
is changed in its negative if we change $(j,k)$ so that the following two quantities
\begin{equation*}
  \alpha j + \beta k \quad \text{and}\quad \sqrt{\alpha\gamma-\beta^2}k,
\end{equation*}
are interchanged, \emph{i.e.} if we replace $(j,k)$ by its image by the symmetry with respect to the bisector of the two straight lines
$\alpha x + \beta y = 0$  and $\sqrt{\alpha\gamma-\beta^2}y =0$ along its orthogonal for the scalar product $q$. Unfortunately, this transformation does not in general send an element of $\ZZ^2$ to another couple of integers. But if we call $(j',k')$ the closest lattice point to the image of $(j,k)$ by this symmetry, then the sum of the contributions of $(j,k)$ and $(j',k')$ is of order $O(||(j,k)||^{-3})$ and thus is summable. If the neighborhood $V$ we took around the origin is invariant under this symmetry, then the sum over all $(j,k)$ in this neighborhood of \eqref{eq:mainterm_der2f} is bounded as $n$ goes to $\infty$.

The fourth derivative of $f$ is a $O(||(x,y)||^{-4})$ and thus
\begin{equation*}
  \frac{1}{n^4}\sup_{x\in\left[ \frac{2j-1}{n},\frac{2j+1}{n} \right]}\left|\frac{\partial^4 f}{\partial x ^4} \left(x,\frac{2k}{n}\right)\right| =O\left( \frac{1}{||(j,k)||^4} \right),
\end{equation*}
which is summable over $\ZZ^2$. The contribution of the error term in \eqref{eq:Taylor_der2f} is bounded as $n$ goes to infinity. All the contributions from the first sum in the RHS of \eqref{eq:ratiodetBbar} are bounded when $n$ gets large. So we get finally that
\begin{equation*}
  \log\left|\frac{\det\bar\Bs}{\mathcal{Z}_n}\right| \leq  \log \left|\frac{\det \bar \Bs}{\det K_n^{10}}\right| = 2\log n + O(1),
\end{equation*}
which ends the proof of Part $2$ of Lemma \ref{lem:firstterm1}.
\end{proof}

\section{Characteristic polynomial and the Laplacian}
\label{sec:P_Harnack}
Let $G$ be an infinite $\ZZ^2$-periodic isoradial graph, and let $\GD$ be the corresponding Fisher graph. 
In this section, we define the {\em critical Laplacian operator} of the graph $G$, the {\em Laplacian characteristic polynomial} and the {\em Laplacian spectral curve}. We also recall the definition of the characteristic polynomial of the dimer model on $\GD$. Part $1$ of Theorem \ref{thm:P_Harnack} states that the Laplacian spectral curve is a Harnack curve of genus $0$ with a unique point $(1,1)$ on the unit torus, it is proved in Section \ref{subsec:lapl}. Part $2$ of Theorem \ref{thm:P_Harnack} states that the dimer spectral curve of $\GD$ and the Laplacian spectral curve of $G$ are the same, it is proved in Section \ref{subsec:kast}.

The critical weights for the Laplacian are $\tan(\theta)$~\cite{Kenyon3}. The {\em Laplacian matrix} on $G$, corresponding to these weights, is defined by: 
\begin{equation*}
  \Delta_{u,v}=
\begin{cases}
  \tan(\theta_{uv})&\text{if $u\sim v$},\\
  -\sum_{u'\sim u}\tan(\theta_{uu'}) &\text{if $u=v$},\\
  0 & \text{otherwise}.
\end{cases}
\end{equation*}
This naturally defines the Laplacian operator acting on functions on $V(G)$.

In a similar way to what we did for the Kasteleyn operator in Section \ref{sec:Kfourier}, let $\widehat{\Delta}(z,w)$ be the Fourier transform of the Laplacian $\Delta$. The {\em Laplacian characteristic polynomial}, denoted $P_\Delta(z,w)$, is the determinant of $\widehat{\Delta}(z,w)$. The set 
\begin{equation*}
\{(z,w)\in\CC^2\,:\,P_\Delta(z,w)=0\}
\end{equation*}
is called the {\em Laplacian spectral curve} of $G$.

Recall the definition of the characteristic polynomial $P(z,w)$ of the critical dimer model on $\GD$:
\begin{equation*}
  P(z,w)=\det(\widehat{K}(z,w)),
\end{equation*}
where $\widehat{K}(z,w)$ is the Fourier transform of the Kasteleyn matrix $K$ of the dimer model on $\GD$. The main result of this section is the following:
\begin{thm}
  \label{thm:P_Harnack}$\,$
  \begin{enumerate}
\item The Laplacian spectral curve of $G$ is a Harnack curve of genus $0$. Its unique point on $\TT^2$ is $(1,1)$, and it has multiplicity $2$.
\item There exists a constant $c\neq 0$, such that 
      \begin{equation*}
	P(z,w)= c P_\Delta (z,w).
      \end{equation*}
\end{enumerate}
\end{thm}

The proof of Theorem \ref{thm:P_Harnack} is given in the next two sections.

\subsection{Laplacian spectral curve: proof of Part $1$ of Theorem \ref{thm:P_Harnack}}\label{subsec:lapl}

In order to prove Part $1$ of Theorem \ref{thm:P_Harnack}, we show in Proposition \ref{prop:detDelta=detKdouble} that the Laplacian characteristic polynomial is equal, up to an explicit multiplicative constant, to the characteristic polynomial of the dimer model on the {\em double of the graph $G$}, denoted $G_d$, defined below. The graph $G_d$ is isoradial and bipartite, so that by Kenyon and Okounkov \cite{KO}, we know that the dimer spectral curve of $G_d$ is a Harnack curve of genus $0$. 

Moreover by definition of the Laplacian, the sum of the columns of the matrix $\widehat{\Delta}(1,1)$ is the zero vector, so that $\det(\widehat{\Delta}(1,1))=P_{\Delta}(1,1)=0$. Since a Harnack curve has $0$ or $2$ conjugate zeros on a torus of given radius, we know that $(1,1)$ is the unique double zero of $P_{\Delta}(z,w)$ on $\TT^2$.

The statement of Proposition \ref{prop:detDelta=detKdouble} is given at the end of this section, since it is more natural to prove it before stating it.
 
\subsubsection{The double of the graph $G$}

The \emph{double of the graph $G$}, denoted by $G_d$, is the bipartite planar graph constructed from $G$  as follows: there is a black vertex in $G_d$ for every vertex of $G$ and every vertex of $G^*$, and there is a white vertex in the middle of each edge of $G$ (or each edge of $G^*$). A white vertex $w$ and a black vertex $b$ are connected by an edge in $G_d$ if the vertex of $G$ or $G^*$ corresponding to $b$ is adjacent to the edge corresponding to $w$, see Figure~\ref{fig:isodiamond}. The double $G_d$ is again a periodic isoradial graph, and its diamond graph is obtained from $G^{\diamond}$ by splitting each rhombus in 4 identical smaller rhombi. 

\begin{figure}[h]
    \begin{center}
      \includegraphics[height=3.8cm]{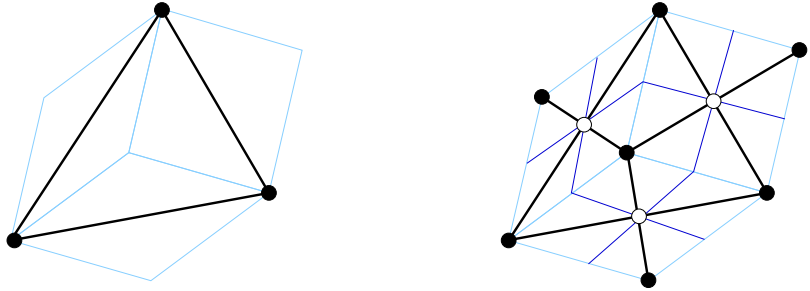}
      \caption{A piece of $G$, together with the rhombi corresponding to the represented edges (left). The corresponding piece in the double $G_d$ (right). The diamond graph of $G_d$ is obtained by splitting every rhombus of $G^{\diamond}$  in four identical smaller rhombi.}
\label{fig:isodiamond}
    \end{center}
\end{figure}
Note that by construction, the double of the dual graph $G^*$ is also $G_d$.

In \cite{Kenyon3}, Kenyon defines a bipartite complex Kasteleyn operator $\mathcal{K}$ on $G_d$ as an infinite matrix whose rows (resp. columns) are indexed by white (resp. black) vertices of $G_d$. If $w$ is a white vertex and $b$ a black vertex of $G_d$, then $\mathcal{K}_{w,b}=0$ when they are not neighbors. When they are adjacent to each other, then 
\begin{equation}
  \mathcal{K}_{w,b}=
  \frac{e^{i\beta}-e^{i\alpha}}{i} = 2\sin \theta_{wb} e^{i\frac{\beta-\alpha}{2}},
\end{equation}
where $e^{i\alpha}$ and $e^{i\beta}$ are the two unit vectors based at $w$ representing the two sides of the rhombus containing the edge $wb$, see Figure \ref{fig:rhombus} below.

\begin{figure}[ht]
  \begin{center}
    \includegraphics[width=3.8cm]{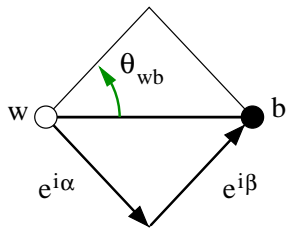}
    \caption{A rhombus of the diamond graph of $G_d$. It contains an edge between a white vertex $w$ and a black vertex $b$. The sides of the rhombus are represented by the two unit complex numbers $e^{i\alpha}$ and $e^{i\beta}$.}
    \label{fig:rhombus}
  \end{center}
\end{figure}

\subsubsection{Kasteleyn operator on $G_d$, incidence matrix and the Laplacian}

Consider any orientation of the edges of $G$. Then, for every white vertex $w$ of $G_d$ (corresponding to an edge of $G$), and for every vertex $b$ of $G$, the coefficient $M^G_{w,b}$ of the {\em oriented weighted incidence matrix} $M^G$, is defined by:
\begin{equation*}
M^{G}_{w,b}= 
\begin{cases}
0 & \text{if the edges corresponding to $w$ is not adjacent to $b$}\\
-\sqrt{\tan\theta} & \text{if the edge corresponding to $w$ is oriented away from $b$}\\
\sqrt{\tan\theta} & \text{if the edge corresponding to $w$ is oriented towards $b$},
\end{cases}
  \end{equation*}
  where $\theta$ is the half-angle of the rhombus of $G^{\diamond}$ containing the edge of $G$ corresponding to $w$. $M^{G^*}$, the incidence matrix of $G^*$ is defined as above, with the orientation of the edges of $G^*$ induced by that of $G$. 

In order to avoid confusion, let us denote by $\Delta^G$ and $\Delta^{G^*}$ the Laplacian on $G$ and $G^*$ respectively. From the definition, we then have:
\begin{equation}
(M^G)^\dagger M^{G}= \Delta^G, \quad (M^{G^*})^\dagger M^{G^*}=\Delta^{G^*},\quad (M^G)^\dagger M^{G^*}=0.
    \label{eq:Delta=MtM}
\end{equation}

Following Kenyon \cite{Kenyon3}, we now relate the Kasteleyn matrix $\mathcal{K}$, the incidence matrices $M^G$ and $M^{G^*}$ and the Laplacians $\Delta^G$ and $\Delta^{G^*}$. Define a diagonal operator $A$ on white vertices of $G_d$ by:
\begin{equation}
  A_{w,w} = \frac{e^{-i\psi}}{2\sqrt{\sin\theta \cos\theta}},
\end{equation}
where $e^{i\psi}$ is the unit vector pointing in the same direction as the oriented edge of $G$ corresponding to $w$, and $\theta$ is the half-angle of the rhombus containing this edge. If we order the columns of $\mathcal{K}$ by listing first the black vertices corresponding to vertices of $G$, and then vertices of $G^*$, the product $A \mathcal{K}$ has the following form:
\begin{equation*}
  A \mathcal{K} = \begin{pmatrix} M^G & i M^{G^*} \end{pmatrix}.
\end{equation*}
Using \eqref{eq:Delta=MtM}, this yields:
\begin{equation}
   {\mathcal{K}}^\dagger A^\dagger  A\mathcal{K}= (A\mathcal{K})^\dagger (A\mathcal{K}) =  
   \begin{pmatrix}  \left(M^G\right)^\dagger \\ -i \left( M^{G^*} \right)^\dagger \end{pmatrix}  
     \begin{pmatrix} M^G & i M^{G^*} \end{pmatrix} = 
	\begin{pmatrix} \Delta^G & 0 \\ 0 & \Delta^{G^*} \end{pmatrix}.
	  \label{eq:Delta=KK*}
\end{equation}
All the above operators commute with translations of $\ZZ^2$. Equation~\eqref{eq:Delta=KK*} induces a relation on the Fourier transforms, which are finite matrices with polynomial entries in the variables $(z,w)\in \TT^2$:
\begin{equation*}
    \widehat{{\mathcal{K}}^\dagger} \widehat{A^\dagger}  \widehat{A} \widehat{\mathcal{K}} =
    \widehat{\mathcal{K}}^\dagger  {\widehat{A}}^\dagger  \widehat{A}  \widehat{\mathcal{K}}
    = \begin{pmatrix} \widehat{\Delta^G} & 0 \\ 0 & \widehat{\Delta^{G^*}}\end{pmatrix}.
  \end{equation*}
Taking the determinant on both sides, and using the fact that:
\begin{equation*}
\left|\det\widehat{A}\right| = \prod_{\substack{w\in V({G_d}_1)\\ w \text{ white}}} \left|A_{w,w}\right|=\prod_{e\in E(G_1)} \frac{1}{2\sqrt{\sin\theta_e\cos\theta_e}},
\end{equation*}
yields the following relation: for every $(z,w)\in\TT^2$,
\begin{equation}
    \label{eq:relKDelta1}
   \left(\prod_{e\in E(G_1)}\frac{1}{4\sin\theta_e\cos\theta_e}\right)
\left\lvert\det \widehat{\mathcal{K}}(z,w)\right\rvert^2= \det \widehat{\Delta^G}(z,w) \det\widehat{\Delta^{G^*}}(z,w).
\end{equation}

The next step consists in relating $\det \widehat{\Delta^G}(z,w)$ and $\det\widehat{\Delta^{G^*}}(z,w)$.
This is done by giving a combinatorial interpretation of $\det \widehat{\Delta^G}(z,w)=P_{\Delta}(z,w)$ in terms of \emph{cycle-rooted spanning forests} on $G_1$. This is very similar to the link between the determinant of a principal minor of the Laplacian of a graph and spanning trees on this graph, given by Kirchhoff's matrix-tree theorem \cite{Kirchhoff:1847}.

\subsubsection{Combinatorial interpretation of $P_{\Delta}(z,w)$ and conclusion of Part $1$}

A \emph{cycle-rooted tree} $T$ of $G_1$ is a connected subgraph of $G_1$ with a unique cycle, which must be non trivial, \emph{i.e.} not contractible to a point on the torus on which $G_1$ is drawn. The \emph{homology class} of $T$ is that of its non trivial cycle in $\ZZ^2$, which is well defined up to a sign since the cycle is not oriented. A \emph{cycle-rooted spanning forest} of $G_1$ is a collection of disjoint cycle-rooted trees covering every vertex of $G_1$. Let us denote by $\mathcal{F}$ the set of cycle-rooted spanning forests of $G_1$.

\begin{lem} 
  \label{lem:detDelta_tree}
The Laplacian characteristic polynomial is the following combinatorial sum:
\begin{equation}
    P_{\Delta}(z,w)= \sum_{F\in\mathcal{F}} \prod_{T\in F} 
\left(\prod_{e\in T} \tan \theta_e \left\lvert \sqrt{z^x w^y}-\sqrt{z^{-x}w^{-y}}\right\rvert^2\right),   \label{eq:weightforest}
  \end{equation}
where $\pm(x,y)\in\ZZ^2$ is the homology class of the cycle-rooted tree $T$.
\end{lem}

\begin{proof}
We follow the same strategy as for the proof of the matrix-tree theorem (see \emph{e.g.} \cite{Biggs}). Since we don't use the dual graph $G^*$ in this argument, in order to improve readability,
we drop the superscript $G$ in the operators $\Delta^G$ and $M^G$, and their Fourier transform.
We start from the Fourier transform of the first relation in \eqref{eq:Delta=MtM}
  \begin{equation*}
    \widehat{\Delta}={\widehat{M}}^{\dagger} \widehat{M}.
  \end{equation*}
Recall that the rows of $\widehat{M}$ are indexed by edges of $G_1$, and its columns are indexed by vertices of $G_1$. Let us use Cauchy-Binet's Theorem to express the determinant of $\widehat{\Delta}$ as
  \begin{equation*}
    \det\widehat{\Delta}(z,w)
    =\sum_{\substack{S\subset E(G_1) \\ |S|=|V(G_1)|}} \det (\widehat{M}^\dagger(z,w)^S) \det (\widehat{M}(z,w)_S)
    =\sum_{\substack{S\subset E(G_1) \\ |S|=|V(G_1)|}} \left|\det(\widehat{M}(z,w)_S)\right|^2.
  \end{equation*}
The matrix $\widehat{M}(z,w)_{S}$ is the incidence matrix of the graph $G_S$ obtained from $G_1$ by taking only the edges in $S$. If $G_S$ contains a trivial cycle, then we can find a function with support on this cycle in the kernel of $M_S$. Therefore the only contributions come from graphs, whose connected components have a number of edges equal to the number of vertices, and each have a non trivial cycle, \emph{i.e.} a cycle-rooted tree. All these non trivial cycle are \emph{parallel}, i.e. they have the same \emph{homology}, up to a sign.
 
After a possible reordering of the rows and of the columns of the matrix $\widehat{M}(z,w)_{S}$, we can assume that it is block diagonal, each diagonal block corresponding to a connected component. The square root of the weight of each edge in $S$ can be factored out from the determinant from each row of $\widehat{M}(z,w)_S$. The determinant of each block can be evaluated recursively by expanding it along columns corresponding to vertices of degree 1, \emph{i.e.} leaves of the tree, collecting a product of terms of modulus $1$ along the branches. The evaluation of the determinant reduced to the cycle gives two terms of modulus $1$, and the ratio between these two terms is $-z^x w^y$.
\end{proof}

A direct consequence of this lemma is:
\begin{equation}
\det \widehat{\Delta^G}(z,w) = \left(\prod_{e\in E(G_1)} \tan \theta_e\right) \det \widehat{\Delta^{G^*}}(z,w).
  \label{eq:DG=DG*}
\end{equation}

Indeed, the dual $T^*$ of a cycle rooted spanning forest $T$ on $G_1$, consisting of edges $e^*$ such that $e\not\in T$, is a cycle-rooted spanning forest on $G_1^*$. Moreover, since the weight of an edge for the Laplacian is the reciprocal of its dual edge,
\begin{equation*}
  \tan \theta_{e^*} =\tan\left( \frac{\pi}{2}-\theta_e \right)=(\tan \theta_e)^{-1},
\end{equation*}
the weight of $T^*$ is equal to the weight of $T$ divided by $\prod_{e\in E(G_1)}\tan \theta_e$.

Combining \eqref{eq:DG=DG*} and \eqref{eq:relKDelta1} yields: 
\begin{prop}
There exists a complex number $\zeta$ of modulus 1, such that for every $(z,w)\in\CC^2$,
  \label{prop:detDelta=detKdouble}
      \begin{equation*}
	\det \widehat{\Delta^G}(z,w)= \zeta \left(\prod_{e\in E(G_1)}2\cos\theta_e\right)^{-1}\det \widehat{\mathcal{K}}(z,w).
      \end{equation*}
\end{prop}

This ends the proof of Part $1$ of Theorem~\ref{thm:P_Harnack}. In the next section we prove Part $2$, namely that the dimer characteristic polynomial $P$ of $\GD$ is equal to $P_{\Delta}$, up to a multiplicative constant.

\subsection{Dimer and Laplacian characteristic polynomial: Proof of Part $2$ of Theorem~\ref{thm:P_Harnack} \label{subsec:kast}}

In order to prove part $2$ of Theorem~\ref{thm:P_Harnack}, we first show in Lemma~\ref{lem:PdivPDelta} that $P(z,w)$ is a divisor of $P_\Delta (z,w)$, and then in Lemma~\ref{lem:incluNewton} that the Newton polygon of $P(z,w)$ is included in the Newton polygon of $P_{\Delta}(z,w)$. Combining these two results yields that there exists a constant $c\neq 0$, such that
\begin{equation*}
  \forall (z,w)\in\CC^2,\quad  \det \widehat{K}(z,w) = c \det \widehat{\Delta}(z,w).
\end{equation*}

\begin{lem}
  \label{lem:PdivPDelta}
  $P(z,w)$ is a divisor of $P_{\Delta}(z,w)$.
\end{lem}

\begin{proof}
  Let $\mathcal{F}^{z,w}(G)$ be the space of $(z,w)$-quasiperiodic functions
  \begin{equation*}
    \mathcal{F}^{z,w}(G)=\left\{ f:\ZZ^2 \rightarrow \CC^{V(G)}\ :\ \forall (x,y)\in\ZZ^2,\ f(x,y)=z^{-x} w^{-y} f(0,0)\right\}.
  \end{equation*}
  A function of this space is entirely determined by its value on a fundamental domain. This space is thus naturally in bijection with $\CC^{V(G_1)}$ by the following projection:
  \begin{align*}
    \pi: \mathcal{F}^{z,w}(G) &\rightarrow \CC^{V(G_1)}\\
    f &\mapsto f(0,0)
  \end{align*}

  Since $\Delta$ is translation invariant, we have for any function $f\in\mathcal{F}^{z,w}(G)$ that
  \begin{equation*}
    \forall (x,y)\in\ZZ^2,\ \forall v\in V(G_1),\quad \bigl(\Delta^G f\bigr)(x,y)_v= z^{-x} w^{-y} \left(\widehat{\Delta^G}(z,w)f(0,0)\right)_v.
    \end{equation*}
    $\widehat{\Delta^G}(z,w)$ can thus be interpreted as the matrix of the restriction of $\Delta^G$ to $\mathcal{F}^{z,w}(G)$ in its canonical basis.

    By \cite{Kenyon3,Mercat}, the \emph{discrete exponential function} on $G$ with parameter $\lambda\in\CC$ is harmonic (\emph{i.e} in the kernel of $\Delta$) and $(z(\lambda),w(\lambda))$-periodic for     
    \begin{equation}
     z(\lambda)=\prod_{\substack{e=\exp(i\alpha)\in E(G^\diamond_1) \\ e \cap \gamma_{1,x}\neq \emptyset}}\frac{\lambda-e^{i\alpha}}{\lambda + e^{i\alpha} },\quad
    w(\lambda)=\prod_{\substack{e=\exp(i\alpha)\in E(G^\diamond_1) \\ e \cap \gamma_{1,y}\neq \emptyset}}\frac{\lambda-e^{i\alpha}}{\lambda + e^{i\alpha} }.
    \label{eq:paramPdelta}
  \end{equation}

    The projection of this exponential function on $G_1$ is a vector in the kernel of the matrix $\widehat{\Delta^G}(z(\lambda),w(\lambda))$. Thus $(z(\lambda),w(\lambda))$ is a point of the spectral curve. The Laplacian spectral curve is Harnack, and thus irreducible. Therefore $\lambda\mapsto \left(z(\lambda),w(\lambda) \right)$ gives a parametrization of the whole curve, except maybe some isolated points.

  But for any $\lambda\in\CC$, we can also construct a $(z(\lambda),w(\lambda))$-periodic function $f_\lambda$ on vertices of $\GD$ such that
  \begin{equation}
    \forall v\in V(\GD),\quad K f_\lambda (v) =0.
  \end{equation}
  See \cite{isoising2} for the details of the construction.
  Therefore, the matrix $\widehat{K}(z(\lambda),w(\lambda))$ is not invertible, and $(z(\lambda),w(\lambda))$ is a point in the spectral curve $\{(z,w)\in \CC^2: P(z,w)=0\}$. This shows that $P_{\Delta}$ is a divisor of the characteristic polynomial $P$.   
\end{proof}

We now have to show that the quotient of these two polynomials is a constant. The idea is to compare the degrees. 
Recall that the Newton polygon $N(P)$ of a polynomial $P(z,w)$ is defined as the convex hull in $\RR^2$ of 
\begin{equation*}
  \left\{ (x,y)\in\ZZ^2 : z^x w^y\text{ is a monomial in }P(z,w)\right\}.
\end{equation*}
We show that the following inclusion holds:
\begin{lem}
  \label{lem:incluNewton}
  \begin{equation*}
    N(P) \subset N(P_{\Delta}).
  \end{equation*}
\end{lem}

A key ingredient to the proof of Lemma~\ref{lem:incluNewton} is the following:

\begin{lem}
  \label{lem:dimer_niloops}
  For every monomial of $P(z,w)=\det\bigl(\widehat{K}(z,w)\bigr)$, one can construct a configuration of non intersecting, non trivial oriented loops on the toroidal isoradial graph $G_1$, with homology equal to the degree of the monomial.
\end{lem}

Before proving Lemma~\ref{lem:dimer_niloops}, let us end the proof of Lemma~\ref{lem:incluNewton}.

\begin{proof}[Proof of Lemma~\ref{lem:incluNewton}]
Take an extremal point $(x_0,y_0)$ of the Newton polygon of $P(z,w)$, then $z^{x_0}w^{y_0}$ is a monomial of $P(z,w)$. Construct a family of $n$ non trivial loops according to Lemma \ref{lem:dimer_niloops}.
Then each loop has homology $\pm(x,y)=\pm\left(\frac{x_0}{d},\frac{y_0}{d}\right)$, where $d=\gcd(x_0,y_0)$. From this set of loops, one can grow a cycle-rooted spanning forest $F$ by gluing new edges step by step. This forest has $n$ trees, and the weight of the forest, computed as in \eqref{eq:weightforest} is a polynomial in $z$ and $w$, with monomials of degree in
\begin{equation}
  D_F=\left\{-n,\dots, n\right\}(x,y).
  \label{eq:convhull_deg}
\end{equation}

The set of possible degrees $D_F$ in the weight of a spanning forest constructed from this family of loops contains $(x_0,y_0)$. Note that the terms with extremal degree come with positive coefficients. It follows that the whole set $D_F$ is included in the Newton polygon of $N(P_{\Delta})$. It might happen that the monomials corresponding to the extremal points of $D_F$ are canceled ``accidentally'' by the contributions of other forests $F'$, but the corresponding monomials in these contributions have to come with a negative coefficient, implying that the extremal degree of these terms is larger, and the corresponding set of possible degrees $D_{F'}$ contains $D_F$.
Thus the extremal point $(x_0,y_0)$, as an element of $D_F$, belongs to $N(P_{\Delta})$.

Since $N(P_{\Delta})$ contains all the extremal points of $N(P)$, it contains all the Newton polygon $N(P)$ by convexity. 
\end{proof}

\begin{proof}[Proof of Lemma \ref{lem:dimer_niloops}]
 In order to construct the family of loops on $G_1$, we need some combinatorial description of $\det\bigl(\widehat{K}(z,w)\bigr)$.  The matrix $\widehat{K}(z,w)$ is not skew-symmetric, but only skew-Hermitian, because of the complex entries $z$ and $w$. Thus, the determinant of $\widehat{K}(z,w)$ is not the square of a combinatorial sum over dimer configurations, that would correspond to the superposition of two independent dimer configurations.
However, we can still expand the determinant as a sum over permutations, and group terms according to the support of their cycles.

\begin{multline}
  \label{eq:devdetK}
  \det \bigl(\widehat{K}(z,w)\bigr) = \sum_{\sigma\in \mathfrak{S}(\GD_1)} \mathrm{sgn}(\sigma) \prod_{v\in V(\GD_1)} \bigl(\widehat{K}(z,w)\bigr)_{v,\sigma v} =\\
  \sum_{\substack{(\Gamma_\l) \text{ partition} \\ \text{of } V(\GD_1)}} \prod_{\l}(-1)^{|\Gamma_\l|-1} \sum_{\substack{\gamma\text{ cycle}\\ \mathrm{supp} \gamma= \Gamma_\l}} \prod_{v\in\Gamma_\l} \bigl(\widehat{K}(z,w)\bigr)_{v,\gamma v}
\end{multline}

One can make a few remarks about this expression:
\begin{enumerate}
  \item Since $\widehat{K}(z,w)$ is an adjacency matrix of $\GD_1$, the only non zero contributions to this sum will come from partitions of $\GD_1$ into disjoint circuits or doubled edges.
  \item If $\Gamma_\l$ is a circuit, the sum over cycles with support $\Gamma_\l$ contains two terms corresponding to the natural cycle $\gamma$ sending a vertex to its successor on the circuit, and its inverse $\gamma^{-1}$.
    When $\Gamma_\l$ is a doubled edge, the sum contains a single term, representing the transposition of the two extremities, which is equal to its inverse.
  \item The degree of $(z,w)$ of the product $\prod_{v\in\Gamma_\l}  \bigl(K(z,w)\bigr)_{v,\gamma v}$ gives the homology in $\mathbb{Z}^2$ of the oriented cycle $\gamma$ around the torus on which $\GD_1$ is drawn. Two cycles $\gamma$ and $\gamma^{-1}$ have opposite homology.
  \item  Because $\widehat{K}(z,w)$ is skew-Hermitian, the coefficients of the monomials for $\gamma$ and $\gamma^{-1}$ will be equal if $|\Gamma_\l|$ is even, and are negative of each other if $|\Gamma_\l|$ is odd. 
  \item It follows from the two previous remarks that trivial loops of odd length do not contribute (the terms from $\gamma$ and $\gamma^{-1}$ cancel each other). Recall that on $\GD_1$ there are triangles around each decoration, which are trivial cycles of odd length. We will for the moment only cancel out the contributions of configurations containing a cycle around a small triangle of a decoration.
  \item The weight corresponding to a transposition is a positive number.
  \item Since they cannot cross, cycles of non trivial homology are parallel: their homology class in $\ZZ^2$ are equal or opposite to one another.
\end{enumerate}

We say that a cycle on $\GD_1$ \emph{visits} a decoration if the sequence of its oriented edges contains successively:
\begin{itemize}
  \item a \emph{long edge} (\emph{i.e.} an edge coming from $G_1$) oriented toward the decoration,
  \item a sequence of distinct edges of the decoration,
  \item another long edge attached to the decoration, but oriented outward the decoration.
  \end{itemize}
A cycle can visit a decoration in two possible directions (clockwise or counterclockwise), depending on which sector of the decoration is taken to connect the two long edges. A cycle can visit several times the same decoration, but always in the same direction.

Take a monomial $a_{x,y}z^x w^y$ of $P(z,w)=\det \widehat{K}(z,w)$. Consider a configuration of cycles on $\GD_1$ contributing to this monomial in the expansion of the determinant \eqref{eq:devdetK}. 

Consider the set of cycles visiting a particular decoration. If these cycles do not all visit the decoration in the same direction, we can construct just by elementary moves around the decoration a cycle configuration also contributing to the same monomial, but with a weight equal to the exact negative of the weight of the initial configuration. Pairing these configurations gives thus a zero contribution to \eqref{eq:devdetK}. 

The construction is done as follows: since not all the cycles visit the decoration in the same direction, then one can find two neighboring cycles with opposite direction. The sector of the decoration separating the two loops is filled with doubled edges with weight $1$, since we already threw away configurations with cycles around triangles. Now reroute the cycles around the triangles, and switch all the doubled edges (see Figure~\ref{fig:anihiloop}). The two cycles either increased or decreased their length by 1 unit, preserving the global signature. The condition on the Kasteleyn orientation implies that the product 
  \begin{equation*}
    \widehat{K}(u_1,u_2) \widehat{K}(u_2,u_3) \widehat{K}(u_3,u_1) = -1,
\end{equation*}
where $u_1,u_2,u_3$ are vertices of the same small triangle. Consequently, when the cycles are rerouted, the product of the entries of the matrix $\widehat{K}$ for one of the two cycles will stay the same, whereas the one for the other cycle will be changed in its negative. Therefore the total weight of this new cycle configuration is exactly the negative of the weight of the initial configuration we started with.

\begin{figure}
  \begin{center}
    \includegraphics[height=4cm]{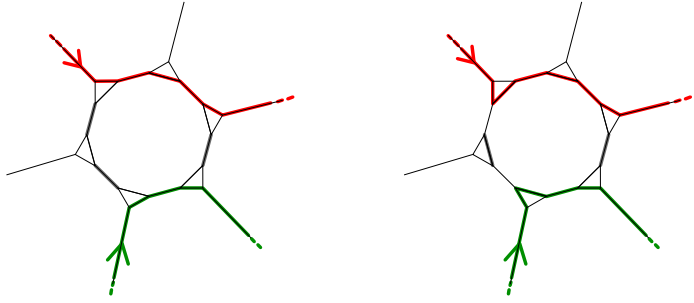}
    \caption{A configuration of two cycles visiting a decoration in different directions: in the initial state (left) and after rerouting (right).}
  \label{fig:anihiloop}
  \end{center}
\end{figure}

Having paired configurations of cycles with annihilating weights, we are left with configurations where every decoration is visited by cycles in the same direction. If there is a decoration visited by two non trivial cycles (there cannot be more, for obvious topological reasons), then they must be parallel.  If they have the same homology class in $\ZZ^2$, they visit the decoration with different orientations. Therefore, they have opposite homology, and thus they do not contribute to the homology class of the cycle configuration. We can erase these two non trivial loops, and search for other decorations on which we can reiterate this operation.

Once that around every decoration,  there is at most one non trivial cycle, one can throw away all the trivial ones as well as the double edges, and retract every decoration to a single vertex. This would lead to a configurations of non intersecting loops on $G_1$ with homology $(x,y)$, which is exactly what we wanted to construct.

\end{proof}

%

%
%

\bibliographystyle{alpha}
\bibliography{ising}

\begin{thebibliography}{{Fis}66}

\bibitem[CR07]{CimaReshe1}
David Cimasoni and Nicolai Reshetikhin.
\newblock Dimers on surface graphs and spin structures. {I}.
\newblock {\em Comm. Math. Phys.}, 275(1):187--208, 2007.

\bibitem[CR08]{CimaReshe2}
David Cimasoni and Nicolai Reshetikhin.
\newblock Dimers on surface graphs and spin structures. {II}.
\newblock {\em Comm. Math. Phys.}, 281(2):445--468, 2008.

\bibitem[dT07]{Bea1}
B{\'e}atrice de~Tili{\`e}re.
\newblock Quadri-tilings of the plane.
\newblock {\em Probab. Theory Related Fields}, 137(3-4):487--518, 2007.

\bibitem[{Fis}66]{Fisher}
M.~E. {Fisher}.
\newblock {On the Dimer Solution of Planar Ising Models}.
\newblock {\em Journal of Mathematical Physics}, 7:1776--1781, October 1966.

\bibitem[Kas67]{Kasteleyn}
P.~W. Kasteleyn.
\newblock Graph theory and crystal physics.
\newblock In {\em Graph {T}heory and {T}heoretical {P}hysics}, pages 43--110.
  Academic Press, London, 1967.

\bibitem[Ken02]{Kenyon3}
Richard Kenyon.
\newblock {The {L}aplacian and {D}irac operators on critical planar graphs}.
\newblock {\em Invent. Math.}, 150(2):409--439, 2002.

\bibitem[KS05]{KeSchlenk}
Richard Kenyon and Jean-Marc Schlenker.
\newblock Rhombic embeddings of planar quad-graphs.
\newblock {\em Trans. Amer. Math. Soc.}, 357(9):3443--3458 (electronic), 2005.

\bibitem[Kup98]{Kuperberg}
Greg Kuperberg.
\newblock An exploration of the permanent-determinant method.
\newblock {\em Electron. J. Combin.}, 5:Research Paper 46, 34 pp. (electronic),
  1998.

\end{thebibliography}


\newcommand{\etalchar}[1]{$^{#1}$}
\begin{thebibliography}{DZM{\etalchar{+}}96}

\bibitem[Bax89]{Baxter}
Rodney~J. Baxter.
\newblock {\em Exactly solved models in statistical mechanics}.
\newblock Academic Press Inc. [Harcourt Brace Jovanovich Publishers], London,
  1989.
\newblock Reprint of the 1982 original.

\bibitem[BdT09]{isoising2}
C.~Boutillier and B.~de~Tili\`ere.
\newblock The critical {$Z$}-invariant ising model via dimers: locality
  properties.
\newblock {\em To appear}, 2009.

\bibitem[Big93]{Biggs}
Norman Biggs.
\newblock {\em Algebraic graph theory}.
\newblock Cambridge Mathematical Library. Cambridge University Press,
  Cambridge, second edition, 1993.

\bibitem[Bou70]{Bourbaki:alg}
N.~Bourbaki.
\newblock {\em \'{E}l\'ements de math\'ematique. {A}lg\`ebre. {C}hapitres 1 \`a
  3}.
\newblock Hermann, Paris, 1970.

\bibitem[CKP01]{CKP}
Henry Cohn, Richard Kenyon, and James Propp.
\newblock {A variational principle for domino tilings}.
\newblock {\em J. Amer. Math. Soc.}, 14(2):297--346 (electronic), 2001.

\bibitem[CR07]{CimaReshe1}
David Cimasoni and Nicolai Reshetikhin.
\newblock Dimers on surface graphs and spin structures. {I}.
\newblock {\em Comm. Math. Phys.}, 275(1):187--208, 2007.

\bibitem[CS06]{CostaSantos}
R.~Costa-Santos.
\newblock Geometrical aspects of the {$Z$}-invariant ising model.
\newblock {\em The European Physical Journal B}, 53(1):85--90, 2006.

\bibitem[dT07]{Bea1}
B{\'e}atrice de~Tili{\`e}re.
\newblock Quadri-tilings of the plane.
\newblock {\em Probab. Theory Related Fields}, 137(3-4):487--518, 2007.

\bibitem[DZM{\etalchar{+}}96]{Russes}
N.~P. Dolbilin, Yu.~M. Zinov{$'$}ev, A.~S. Mishchenko, M.~A. Shtan{$'$}ko, and
  M.~I. Shtogrin.
\newblock {Homological properties of two-dimensional coverings of lattices on
  surfaces}.
\newblock {\em Funktsional. Anal. i Prilozhen.}, 30(3):19--33, 95, 1996.

\bibitem[{Fis}66]{Fisher}
M.~E. {Fisher}.
\newblock {On the Dimer Solution of Planar Ising Models}.
\newblock {\em Journal of Mathematical Physics}, 7:1776--1781, October 1966.

\bibitem[{Kas}61]{Kast61}
P.~W. {Kasteleyn}.
\newblock {The statistics of dimers on a lattice : I. The number of dimer
  arrangements on a quadratic lattice}.
\newblock {\em Physica}, 27:1209--1225, December 1961.

\bibitem[Kas67]{Kasteleyn}
P.~W. Kasteleyn.
\newblock Graph theory and crystal physics.
\newblock In {\em Graph {T}heory and {T}heoretical {P}hysics}, pages 43--110.
  Academic Press, London, 1967.

\bibitem[Ken97]{Ke:LocStat}
Richard Kenyon.
\newblock {Local statistics of lattice dimers}.
\newblock {\em Ann. Inst. H. Poincar\'e Probab. Statist.}, 33(5):591--618,
  1997.

\bibitem[Ken02]{Kenyon3}
Richard Kenyon.
\newblock {The {L}aplacian and {D}irac operators on critical planar graphs}.
\newblock {\em Invent. Math.}, 150(2):409--439, 2002.

\bibitem[{Kir}47]{Kirchhoff:1847}
G.~{Kirchhoff}.
\newblock {Ueber die Aufl{\"o}sung der Gleichungen, auf welche man bei der
  Untersuchung der linearen Vertheilung galvanischer Str{\"o}me gef{\"u}hrt
  wird}.
\newblock {\em Annalen der Physik}, 148:497--508, 1847.

\bibitem[KO06]{KO}
Richard Kenyon and Andrei Okounkov.
\newblock Planar dimers and {H}arnack curves.
\newblock {\em Duke Math. J.}, 131(3):499--524, 2006.

\bibitem[KOS06]{KOS}
Richard Kenyon, Andrei Okounkov, and Scott Sheffield.
\newblock Dimers and amoebae.
\newblock {\em Ann. of Math. (2)}, 163(3):1019--1056, 2006.

\bibitem[KS05]{KeSchlenk}
Richard Kenyon and Jean-Marc Schlenker.
\newblock Rhombic embeddings of planar quad-graphs.
\newblock {\em Trans. Amer. Math. Soc.}, 357(9):3443--3458 (electronic), 2005.

\bibitem[KW41]{KramersWannier}
H.~A. Kramers and G.~H. Wannier.
\newblock Statistics of the two-dimensional ferromagnet. {I}.
\newblock {\em Phys. Rev. (2)}, 60:252--262, 1941.

\bibitem[Mer01]{Mercat}
Christian Mercat.
\newblock {Discrete {R}iemann surfaces and the {I}sing model}.
\newblock {\em Comm. Math. Phys.}, 218(1):177--216, 2001.

\bibitem[MW73]{McCoyWu}
B.~McCoy and F.~Wu.
\newblock {\em {The two-dimensional Ising model}}.
\newblock Harvard Univ. Press, 1973.

\bibitem[RG57]{GradsteinRyshik}
I.~M. Ryshik and I.~S. Grad{\v{s}}te{\u\i}n.
\newblock {\em Summen-, {P}rodukt- und {I}ntegral-tafeln}.
\newblock VEB Deutscher Verlag der Wissenschaften, Berlin, 1957.

\bibitem[She05]{Sheffield}
Scott Sheffield.
\newblock Random surfaces.
\newblock {\em Ast\'erisque}, (304):vi+175, 2005.

\bibitem[Tes00]{Tesler}
Glenn Tesler.
\newblock {Matchings in graphs on non-orientable surfaces}.
\newblock {\em J. Combin. Theory Ser. B}, 78(2):198--231, 2000.

\bibitem[Wan45]{Wannier}
G.~H. Wannier.
\newblock The statistical problem in cooperative phenomena.
\newblock {\em Rev. Mod. Phys.}, 17(1):50--60, Jan 1945.

\end{thebibliography}

\end{document}